\documentclass[11pt,reqno]{amsart}
\usepackage[margin=1in]{geometry}
\usepackage[fleqn,tbtags]{mathtools}
\usepackage[shortlabels]{enumitem}
\usepackage{graphicx}
\usepackage{amssymb}
\usepackage{amsmath}
\usepackage{mlmodern}
\usepackage{eucal}
\usepackage{tabu}
\usepackage{makecell}
\usepackage[labelformat=simple]{subcaption}
\usepackage{multirow}
\usepackage{graphicx}

\usepackage{hyperref,xcolor}
\hypersetup{
    colorlinks,
    linkcolor={red!50!black},
    citecolor={blue!50!black},
    urlcolor={blue!80!black}
}

\let\oh=\circ
\newcommand{\ccirc}{\mathbin{\mathchoice
  {\xcirc\scriptstyle}
  {\xcirc\scriptstyle}
  {\xcirc\scriptscriptstyle}
  {\xcirc\scriptscriptstyle}
}}
\newcommand{\xcirc}[1]{\vcenter{\hbox{$#1\oh$}}}
\let\circ\ccirc

\newcommand\restr[2]{{
  \left.\kern-\nulldelimiterspace 
  #1 
  \vphantom{\big|} 
  \right|_{#2} 
  }}

\DeclareMathOperator{\diag}{diag}
\DeclareMathOperator{\GL}{\mathsf{GL}}
\DeclareMathOperator{\SL}{\mathsf{SL}}
\DeclareMathOperator{\SO}{\mathsf{SO}}
\DeclareMathOperator{\SU}{\mathsf{SU}}
\DeclareMathOperator{\Sp}{\mathsf{Sp}}
\DeclareMathOperator{\Un}{\mathsf{U}} 
\DeclareMathOperator{\UT}{\mathsf{UT}}
\DeclareMathOperator{\BUT}{\mathsf{BUT}}
\let \P \undefined
\DeclareMathOperator{\P}{\mathsf{P}}
\let \O \undefined
\DeclareMathOperator{\O}{\mathsf{O}}
\let \S \undefined
\DeclareMathOperator{\S}{\mathsf{S}}

\DeclareMathOperator{\rank}{rank}
\DeclareMathOperator{\tr}{tr}
\DeclareMathOperator{\Fl}{Fl}
\DeclareMathOperator{\Gr}{Gr}
\DeclareMathOperator{\Alt}{\mathsf{\Lambda}}
\DeclareMathOperator{\Sym}{\mathsf{S}}
\DeclareMathOperator{\End}{End}
\DeclareMathOperator{\Stiefel}{V}
\DeclareMathOperator{\St}{St}

\newcommand{\G}{\mathsf{G}}
\let \H \undefined
\newcommand{\H}{\mathsf{H}} 

\theoremstyle{definition}
\newtheorem{theorem}{Theorem}[section]
\newtheorem{definition}[theorem]{Definition}
\newtheorem{lemma}[theorem]{Lemma}
\newtheorem{corollary}[theorem]{Corollary}
\newtheorem{proposition}[theorem]{Proposition}
\newtheorem{example}[theorem]{Example}
\numberwithin{equation}{section}

\newcommand{\tp}{{\scriptscriptstyle\mathsf{T}}}
\newcommand{\MP}{{\scriptscriptstyle\mathsf{MP}}}

\newcommand{\C}{\mathbb{C}}
\newcommand{\F}{\mathbb{F}}

\newcommand{\M}{\mathcal{M}}
\newcommand{\N}{\mathbb{N}}
\newcommand{\R}{\mathbb{R}}
\newcommand{\V}{\mathbb{V}}
\newcommand{\W}{\mathbb{W}}
\newcommand{\U}{\mathbb{U}} 

\renewcommand{\sl}{\mathfrak{sl}}
\newcommand{\so}{\mathfrak{so}}
\newcommand{\su}{\mathfrak{su}}
\renewcommand{\sp}{\mathfrak{sp}}
\newcommand{\g}{\mathfrak{g}}

\setlist[description]{font=\normalfont\itshape,labelindent=2em}

\begin{document}
\title{Linear representations of manifolds}
\author[R.~Wang]{Rongbiao Thomas Wang}
\author[L.-H.~Lim]{Lek-Heng~Lim}
\address{Computational and Applied Mathematics, University of Chicago, Chicago, IL 60637}
\email{rbwang@uchicago.edu, lekheng@uchicago.edu}

\author[K.~Ye]{Ke Ye}
\address{KLMM, Academy of Mathematics and Systems Science, Chinese Academy of Sciences, Beijing 100190, China}
\email{keyk@amss.ac.cn}

\begin{abstract}
A finite-dimensional linear representation of a group or an algebra may be regarded as a map into a space of matrices, endowing abstract elements with coordinates, and encoding algebraic operations as matrix products. With this in mind, we define a linear representation of a $\G$-manifold $\M $ as a map into a space of matrices, representing points as matrices and the $\G$-action as matrix products. We show that this generalizes group representations to any $\G$-manifold that may not have a group structure, with homogeneous spaces $\G/\H$ an important special case; and in this case it also generalizes Cartan embeddings of symmetric spaces to more general $\G/\H$. To demonstrate the utility of such manifold representations, we use them to provide effective bounds for Mostow--Palais $\G$-equivariant embeddings of $\G$-manifolds into $\G$-modules $\V$. Unlike Whitney and Nash embeddings, Mostow--Palais embeddings have no known effective bounds; before our work, it was only known that $\dim \V < \infty$ if $\G$ is compact. We will give explicit values for $\dim \V$ and show that our bounds are sharp. Furthermore, our method is constructive, giving explicit expressions for these minimal-dimensional Mostow--Palais embeddings.
\end{abstract}

\maketitle

\section{Introduction}\label{sec:intro}

It is easiest to capture the spirit of what we call ``manifold representation'' through analogy: At their most basic level, assuming finite dimension and working over $\mathbb{R}$, a group representation is a map $\pi : \G \to \GL_n(\R) \subseteq \mathbb{R}^{n \times n}$, an algebra representation $\rho : \mathfrak{g} \to \mathfrak{gl}_n(\R) = \mathbb{R}^{n \times n}$, and a manifold representation $\varepsilon : \M  \to \U \subseteq \mathbb{R}^{n \times n}$. Here $\G$ is a (Lie) group, $\mathfrak{g}$ a (Lie) algebra, $\M $ a $\G$-manifold, $\U$ a $\G$-module, and $\mathbb{R}^{n \times n}$ the $n \times n$ matrices. The maps $\pi$, $\rho$, $\varepsilon$ should respect the inherent algebraic structure, so $\pi$ is a group homomorphism, $\rho$ an algebra homomorphism, and $\varepsilon$ an equivariant map.

Studies of $\G$-equivariant maps of $\G$-manifolds into $\G$-modules are not new. Mostow \cite{Mostow1957_equivariant} and Palais \cite{Palais1957_imbedding} famously investigated $\G$-equivariant embeddings of $\G$-manifolds into Euclidean $\G$-modules seventy years ago. Studies of $\G$-equivariant immersion \cite{Bierstone} and $\G$-equivariant submersions \cite{Lashof} are also common in topology. What is new here is that our $\G$-module is a space of matrices, i.e., $\varepsilon$ is the analogue of a \emph{linear} representation in the context of group and algebra representations.

At first glance, it may appear that there is hardly any difference. As Euclidean spaces, the space of matrices $\mathbb{R}^{n \times n}$ is no different from $\mathbb{R}^{n^2}$, but this is only because we limited ourselves to metric and vector space structures. Matrices are endowed with vastly richer properties --- we may multiply or decompose them; impose orthogonality or symmetry on them; define determinant, norm, rank, or eigen/singular values and vectors; among a myriad of yet other features. We have used this to great effect in our earlier works, resolving conjectures \cite{deg}, demonstrating computational intractability \cite{NP}, calculating complicated geometric objects \cite{curv}, finding novel generators \cite{LLY2025}, among other things --- all through embedding manifolds as submanifolds of $\mathbb{R}^{n \times n}$ in an equivariant manner.

Our notion of a linear representation of a $\G$-manifold $\M $ is consistent with $\G$-equivariant maps of manifolds in differential geometry and linear representations of groups in representation theory. In fact, it extends the latter: Every group representation is a manifold representation with $\M  = \G$. It is also consistent with Cartan immersions of symmetric spaces, which lies in the intersection of differential geometry and representation theory.

\subsection{What's new}\label{sec:new}

Viewing $\G$-equivariant maps of manifolds in light of group representation theory leads us to notions of  equivalence, faithfulness, irreducibility, indecomposability, minimality for manifold representations, and various constructions of new manifold representations from old ones, analogous to those in group representations. These notions and constructions are standard in group representation theory but largely absent from differential geometry. In this article, we show how they can be useful in at least one setting: Obtaining effective, even sharpest, bounds for the dimensions of Mostow--Palais embeddings.

Note that while Whitney embeddings and Nash embeddings of manifolds have well-known effective bounds \cite{Whitney44,Nash56,Gromov86}, Mostow--Palais embeddings do not. Neither the original works of Mostow \cite{Mostow1957_equivariant} and Palais \cite{Palais1957_imbedding} nor subsequent works by them and others are effective. To be clear, a \emph{Mostow--Palais embedding} is a $\G$-equivariant embedding of a $\G$-manifold in an Euclidean $\G$-space $\V$. Before the work in this article\footnote{We also have an unpublished preliminary version \cite{Lim_Ye} that discussed only the case of real flag manifolds.}  it was only known $\V$ can be chosen to be finite-dimensional if $\G$ is compact.

We will determine the sharpest bounds for the dimensions of Mostow--Palais embeddings for some of the most ubiquitous manifolds (see Table~\ref{tab:homo} for a partial list). This is achieved by undertaking the perspective of ``manifold representation theory,'' developed in Sections~\ref{sec:man} and \ref{sec:homo}, where it becomes natural to consider \emph{faithful linear representations} with $\V \subseteq \F^{n \times n}$, $\F = \R$ or $\C$,  and $\G$ acting via matrix multiplication, i.e., congruence, similarity, or equivalence of matrices. A faithful linear representation is of course also a Mostow--Palais embedding but, just as not all Lie groups have faithful linear representations, we will show that having a faithful linear representation is a special property for a $\G$-manifold. In contrast, any $\G$-manifold has a Mostow--Palais embedding when $\G$ is compact. Our main effort is to provide,  for $\G = \SL_n(\F)$, $\SO_n(\F)$, $\Sp_n(\F)$, $\SU_n$, $\SO_{p,q}$, $\Sp_n$,
\begin{enumerate}[\upshape (i)]
\item\label{it:class} a complete list of faithfully linearly representable  $\G$-manifolds $\M $ with a transitive $\G$-action (Theorems~\ref{thm:min}\ref{it:f1} and \ref{thm:mat});
\item\label{it:des} an explicit description of the faithful linear representation in \ref{it:class}, realizing $\M $ as a submanifold of $\mathbb{R}^{n \times n}$ or $\mathbb{C}^{n \times n}$ (Tables~\ref{tab:min} and \ref{tab:mat});
\item a proof that the faithful linear representation in \ref{it:des} is a minimal-dimensional Mostow--Palais embedding (Theorem~\ref{thm:min}\ref{it:f2} and Corollary~\ref{cor:mp}).
\end{enumerate}
These results may be viewed as the manifold analog of minimal-dimensional faithful linear representations of groups \cite{mingrp1,mingrp2} and algebras \cite{minalg2,minalg3,minalg1}.  Our focus on faithful representations in this first foray into manifold representations is not accidental: Historically, group representations and algebra representations also started from faithful representations of these objects \cite{history2, history1}. 

\begin{table}[htb]
\centering
\footnotesize
\tabulinesep=0.2ex
\begin{tabu}{c|c|c|c}
\multicolumn{2}{c|}{$\G$-manifold $\M $}  & homogeneous space $\G/\H$ & type\\
\hline
Grassmannian & \makecell[c] {real\\
complex\\
quaternionic\\
real symplectic\\
complex symplectic\\
complex locus\\
special Lagrangian\\
complex Lagrangian\\
skew-Hermitian Lagrangian\\
orthogonal Lagrangian\\
isotropic\\
indefinite} & \makecell[c] {$\SO_n(\R)/\mathsf{S}( \O_k(\R) \times \O_{n-k}(\R))$\\
$\SU_n/\mathsf{S}(\Un_k \times \Un_{n-k} )$\\
$\Sp_{2n}/(\Sp_{2k} \times \Sp_{2n - 2k})$\\
$\Sp_{2n}(\R)/(\Sp_{2k}(\R) \times \Sp_{2n - 2k}(\R))$\\
$\Sp_{2n}(\C)/(\Sp_{2k}(\C) \times \Sp_{2n - 2k}(\C))$\\
$\SO_n(\C)/\mathsf{S}(\O_{k}(\C)\times \O_{n-k}(\C))$\\
$\SU_n/\SO_n(\R)$\\
$\Sp_{2n}/\Un_n$\\
$\SU_{2n}/\Sp_{2n}$\\
$\SO_{2n}(\R)/\Un_{n}$\\
$\SO_{n}(\R)/\bigl( \Un_{k} \times \SO_{n-2k}(\R) \bigr)$\\
$\SO_{m,n}/\mathsf{S} ( \O_{p,q} \times \O_{m-p,n-q} )$} & \makecell[c] {BDI\\AIII\\CII\\ \, \\ \, \\ \, \\ AI \\ CI \\ AII \\ DIII \\ \\\,}\\
\hline
Flag manifold & \makecell[c] {real\\
complex\\
quaternionic\\
isotropic\\
partial isotropic\\
real symplectic\\
complex symplectic\\
complex Lagrangian} & \makecell[c] {$\SO_n(\R)/\mathsf{S}( \O_{n_1}(\R)\times \dots \times \O_{n_{m+1}}(\R) )$\\
$\SU_n/\mathsf{S}(\Un_{n_1}\times \dots \times \Un_{n_{m+1}})$\\
$\Sp_{2n}/(\Sp_{2n_1}\times \dots \times \Sp_{2n_{m+1}})$\\
$\SO_{2n}(\R)/(\Un_{n_1}\times \dots \times \Un_{n_{m+1}})$\\
$\SO_{2n+p}(\R)/(\Un_{n_1}\times \dots \times \Un_{n_{m+1}} \times \SO_{p}(\R))$\\
$\Sp_{2n}(\R)/(\Sp_{2n_1}(\R) \times \dots \times \Sp_{2n_{m+1}}(\R))$\\
$\Sp_{2n}(\C)/(\Sp_{2n_1}(\C) \times \dots \times \Sp_{2n_{m+1}}(\C))$\\
$\SO_{2n}(\R)/(\Un_{n_1} \times \dots \times \Un_{n_{m+1}})$}\\\hline
Stiefel manifold & \makecell[c] {real\\
complex\\
quaternionic\\
noncompact real\\
noncompact complex} & \makecell[c] {$\SO_n(\R)/\SO_{n-k}(\R)$\\
$\SU_n/\SU_{n-k}$\\
$\Sp_{2n}/\Sp_{2n-2k}$\\
$\SL_n(\R)/\P(I_k,\SL_{n-k}(\R))$\\
$\SL_n(\C)/\P(I_k,\SL_{n-k}(\C))$}
\end{tabu}
\caption{Some common manifolds that are faithfully linearly representable. See Table~\ref{tab:mat} for their minimal faithful linear representations.} 
\label{tab:homo}
\end{table}

Table~\ref{tab:homo} shows some common manifolds in \ref{it:class}. Any  $\G$-manifold $\M $ on which $\G$ acts transitively must be diffeomorphic to a homogeneous space $\G/\H$ for some proper closed subgroup $\H$; we give this alternative characterization in the third column. Note that these homogeneous spaces have no group structure in general since $\H$ is not necessarily normal. For this reason, we may also view the work in this article as an extension of group representation theory to include homogeneous spaces $\G/\H$ --- group representation is the special case when $\H = \{I\}$ (Proposition~\ref{prop:group}). For those homogeneous spaces that happen to be symmetric spaces (of type indicated in the fourth column), their Cartan embeddings are all minimal faithful linear representations (Corollary~\ref{cor:car}).

Table~\ref{tab:homo} is not intended to exhaust the list in \ref{it:class}; some other less common manifolds that are faithfully linearly representable may be found in Table~\ref{tab:min}.

\subsection{Why different}

Unlike Mostow--Palais embeddings, we do not need to assume compactness of $\G$ for our classification of faithful linear representations in \ref{it:class}.

The classification in \ref{it:class} cannot be derived from any known classifications of, say, symmetric spaces (a la Cartan) or finite-type quivers (a la Gabriel). Indeed, most of the manifolds in Table~\ref{tab:homo}  are not  even symmetric spaces. While quivers are implicit when we consider classical groups of types $\mathrm{A}$, $\mathrm{B}$, $\mathrm{C}$, and $\mathrm{D}$, they are clearly not all of finite type, and include tame and wild ones as well. The symmetric spaces in Table~\ref{tab:homo} are Gelfand pairs \cite{Gelfand} but our classification list bears little relation with these: Except in some degenerate cases, the Stiefel manifolds in our list are not Gelfand pairs; on the other hand, oriented Grassmannians are Gelfand pairs but do not appear on our list.

As previously mentioned, any $\G$-manifold with a transitive $\G$-action is diffeomorphic to a homogeneous space $\G/\H$. This invites comparisons with induced representations  \cite{Mackey52,Mackey53} and isotropy representations \cite[Section~II.3]{Helgason2001}. For the former, the critical difference is that while our linear representation of a manifold $\G/\H \to \V$ trivially yields a linear representation of the group $\G$ on $\V$, an induced linear representation of $\G$ is obtained from a linear representation of $\H$. For the latter, our linear representation of $\G/\H$ represents the homogeneous space $\G/\H$ as a submanifold of matrices whereas an isotropy representation represents the subgroup $\H$ on the tangent space at the identity of $\G/\H$; there is no relation between these two kinds of representation in general.

Given that $\F^{n \times n}$ is an \emph{affine} space, our faithful linear representations are distinct from Chevalley's construction \cite[Section~11.2]{Humphreys1975} which, for a linear algebraic group $\G$,  embeds a $\G$-manifold into \emph{projective} $\mathbb{P}(\V)$ for some $\G$-module $\V$.

The upshot is that despite the abundance of adjacent topics, the linear representations of $\G$-manifolds in this article likely constitute a hitherto unexplored subject.

\section{Notations and conventions}

We denote the nonnegative integers by $\N$ and the quaternions by $\mathbb{H}$. As already noted, $\F$ means $\R$ or $\C$. Any statement in $\mathbb{F}$ is understood to hold for both $\R$ and $\C$. We denote elements of $\G$, $\M$, $\V$ by $g$, $x$, $v$; but in specific cases, we follow the convention below in Sections~\ref{sec:mat}, \ref{sec:gp}, and \ref{sec:rep}.

\subsection{Matrices}\label{sec:mat}

For any $k \le n$, we regard $\mathbb{F}^{n \times k} \subseteq \mathbb{F}^{n \times n}$ by identifying $\mathbb{F}^{n \times k}$ with the subspace $\{ X \in \mathbb{F}^{n \times n} : x_{ij} = 0 $ whenever $j > k \}$. We also identify $\mathbb{F}^n \equiv \mathbb{F}^{n \times 1}$, i.e., as the space of column vectors. We write $e_i$ for the standard basis vector in $\mathbb{F}^n$ with $1$ in the $i$th coordinate and $0$ elsewhere. We denote a block diagonal matrix with $X_1,\dots, X_m$ on the diagonal by $\diag (X_1,\dots,  X_m)$. We write $X^{-\tp} \coloneqq (X^{-1})^\tp = (X^\tp)^{-1}$ for any invertible $X \in \F^{n \times n}$.

For special matrices, we use specific letters for easy identification and distinction: $A$ for an invertible matrix; $Q$, $U$, $V$ for matrices in the orthogonal, unitary, or other bilinear/sesquilinear form-preserving subgroups of $\GL_n(\F)$; $X$, $Y$, $Z$ for symmetric/Hermitian, skew-symmetric/skew-Hermititian matrices in  $\F^{n \times n}$ under different bilinear/sesquilinear forms.

We write $B$ for a matrix that defines a symmetric bilinear form, and $\Omega$ for one that defines a skew-symmetric bilinear form. We will not distinguish between a bilinear form and its matrix representation with respect to the standard basis.

\subsection{Groups}\label{sec:gp}

For $\G$ a Lie subgroup of $\GL_n(\F)$, we write
\begin{equation}\label{eq:S}
\S(\G) \coloneqq \G \cap \SL_n(\F) =  \{A \in \G: \det(A) = 1 \}.
\end{equation}
As is customary, we write $\O_n(\F)$, $\SO_n(\F)$, $\Sp_n(\F)$, $\SL_n(\F)$, $\GL_n(\F)$, indicating the field. But we leave out the field in $\O_{p,q}$, $\SO_{p,q}$, $\Un_n$, $\SU_n$. A departure from this rule is the compact symmetric group, defined as
\[
\Sp_{2n} \coloneqq \Sp_{2n}(\C) \cap \SU_{2n} \cong \Un_n(\mathbb{H}),
\]
and is neither $\Sp_{2n}(\R)$ nor $\Sp_{2n}(\C)$, but isomorphic to the group of quaternionic unitary matrices.

Clearly, each of these groups has infinitely many isomorphic copies within $\GL_n(\F)$. We emphasize that we regard them as concrete matrix subgroups of $\GL_n(\F)$ defined in the standard ways; and we distinguish between isomorphic copies by  specifying the bilinear forms. So 
\begin{align*}
\SO_n(\C; B) &\coloneqq \{Q \in \GL_n(\C) : Q^\tp B Q = B, \, \det(Q) = 1\}  \cong \SO_n(\C)
\intertext{for any symmetric matrix $B \in \GL_n(\C)$ and}
\SO_n(\R; B) &\coloneqq \{Q \in \GL_n(\R) : Q^\tp B Q = B, \, \det(Q) = 1\} \cong \SO_{p,q}
\intertext{for any symmetric matrix $B \in \GL_n(\R)$ with $(p,q)$ the inertia of $B$. Similarly,}
\Sp_{2n}(\F; \Omega) &\coloneqq \{Q \in \GL_{2n}(\F) : Q^\tp \Omega Q = \Omega, \, \det(Q) = 1\} \cong \Sp_{2n}(\F),\\
\Sp_{2n}(\Omega) &\coloneqq \Sp_{2n}(\C; \Omega) \cap \SU_{2n} \cong \Sp_{2n}
\end{align*}
for any  skew-symmetric $\Omega \in \GL_{2n}(\F)$. So $\SO_n(\F) = \SO_n(\F; I_n)$ and $\Sp_{2n}(\F) = \Sp_{2n}(\F; J_{2n})$, where $I_n$ is the $n \times n$ identity matrix and $J_{2n} \coloneqq \bigl[\begin{smallmatrix}
    0 &I_n\\
    -I_n &0
\end{smallmatrix}\bigr] \in \F^{2n \times 2n}$. We would not study the indefinite unitary and symplectic groups $\Un_{p,q}$,  $\SU_{p,q}$, $\Sp_{p,q}$ in this article (see Section~\ref{sec:red}).

We remind readers that the \emph{type} of a homogeneous space associated with a classical group \cite[Section~1.7.2]{Goodman_Wallach} is used in the following sense:
\begin{description}
    \item [Type~$\mathrm{A}$] $\SL_n(\C)$ with its split and compact real forms $\SL_n(\R)$ and $\SU_n$.
    \item [Types~$\mathrm{B}$ and $\mathrm{D}$] $\SO_n(\C)$ with its split and compact real forms $\SO_{p,q}$ and $\SO_n(\R)$. 
    \item [Type~$\mathrm{C}$] $\Sp_{2n}(\C)$ with its split and compact real forms $\Sp_{2n}(\R)$ and $\Sp_{2n}$. 
\end{description}

\subsection{Group representations}\label{sec:rep}

For unambiguous terminologies, we state them here: Let $\G$ be a Lie subgroup of $\GL_n(\F)$. A linear \emph{representation} of $\G$ is a group homomorphism $\pi: \G \to \GL(\V)$ where $\V$ is an $\F$-vector space. Clearly,  $\V$ is a $\G$-module with an action induced by $\pi$. Henceforth, we reserve the letter $\V$ for a $\G$-module when the ground field $\F$ is unspecified, $\U$ for when $\F = \R$, and $\W$ for when $\F = \C$. For a $\G$-module $\V \subseteq \mathbb{F}^{n \times n}$, we denote the subspace of traceless matrices in $\V$ by
\[
\V_\oh \coloneqq \{X \in \V: \tr(X) = 0\}.
\]
We summarize the common $\G$-modules in our article and their dimensions in Table~\ref{tab:vec}.
\begin{table*}[htb]
\centering
\footnotesize
\tabulinesep=0.75ex
\begin{tabu}{@{}l|l|l}
vector space &definition
& dimension over $\F$ \\
\hline
$\F^{n \times k}$ & $\{X \in \F^{n \times n} : x_{ij} = 0 \text{ for all } j > k \}$ &$nk$\\
$\Alt^2(\F^n)$ &$\{X \in \F^{n \times n} : X^\tp = -X\}$ & $\frac{1}{2}n(n-1)$\\
$\Sym^2_\oh(\F^n)$ &$\{X \in \F^{n \times n} : X^\tp = X, \,\tr(X)= 0\}$ & $\frac{1}{2}(n+2)(n-1)$ \\
$\Sym^2(\F^n)$ &$\{X \in \F^{n \times n} : X^\tp = X\}$ & $\frac{1}{2}n(n+1)$ \\
$\sl_n(\F) = \F^{n \times n}_\oh$ &$\{X \in \F^{n \times n} : \tr(X)= 0\}$ &$n^2-1$\\
$\su_n$ &$\{X \in \C^{n \times n} : X^* = -X, \, \tr(X) = 0\}$ & $n^2-1$ $(\F = \R)$\\
$\mathfrak{u}_n$ &$\{X \in \C^{n \times n} : X^* = -X\}$ & $n^2$ $(\F = \R)$\\
$\H_\oh^2(\C^n)$ &$\{X \in \C^{n \times n} : X^* = X, \, \tr(X) = 0\}$ &$n^2-1$ $(\F = \R)$\\
$\Alt^2(\R^n;I_{p,q})$ &$\{X \in \R^{n \times n}: X^\tp I_{p,q} = - I_{p,q}X\} $ & $\frac{1}{2}n(n-1)$ $(\F = \R)$\\
$\Sym_\oh^2(\R^n;I_{p,q})$ & $\{X \in \R^{n \times n}: X^\tp I_{p,q} =  I_{p,q}X,\, \tr(X) = 0\}$ & $\frac{1}{2}(n+2)(n-1)$ $(\F = \R)$\\
$\Alt^2(\F^{2n};\Omega)$ & $\{X \in \F^{2n \times 2n}: X^\tp \Omega =  -\Omega X\}$ & $2n^2+n$\\
$\Sym_\oh^2(\F^{2n};\Omega)$ & $\{X \in \F^{2n \times 2n}: X^\tp \Omega =  \Omega X,\, \tr(X) = 0\}$ &$(n-1)(2n+1)$\\
$\sp_{2n}$ & $\su_{2n} \cap \Alt^2(\C^{2n};\Omega)$ & $2n^2+n$ $(\F = \R)$\\
$\Alt^k(\F^n)$ & $(\F^n)^{\otimes k}/\langle\,v_1\otimes\dots\otimes v_k : v_i=v_j\text{ for some }i<j\rangle$ & $\binom{n}{k}$
\end{tabu}
\caption{Dimensions of group representations.}
\label{tab:vec}
\end{table*}

In Table~\ref{tab:vec}, for a nondegenerate bilinear form $B \in \GL_n(\F)$, we write
\[
\Sym^2(\F^n; B) = \{X \in \F^{n \times n}: X^\tp B =  BX\}, \quad \Alt^2(\F^n; B) = \{X \in \F^{n \times n}: X^\tp B =  -BX\}, 
\]
simplified to $\Sym^2(\F^n)$ and $\Alt^2(\F^n)$ when $B=I_n$. For $B = I_n$,  $I_{p,q}$, and $ J_{2n}$, note that $\Alt^2(\F^n; B) = \so_n(\F)$, $\so_{p,q}$, and $\sp_{2n}(\F)$ respectively.

We remind readers of the Weyl dimension formulae \cite{WeylCombined}, explicitly calculated in \cite[Exercises~7.1.8--7.1.11]{Goodman_Wallach} for $\mathfrak{g} = \mathfrak{sl}_n(\C)$,  $\mathfrak{so}_n(\C)$, and $\mathfrak{sp}_n(\C)$. Recall that for these choices of $\mathfrak{g}$, the irreducible complex $\mathfrak{g}$-modules are indexed by $m$-tuples $(\kappa_1,\dots,  \kappa_m) \in \N^m$ where 
\begin{equation}\label{eq:m}
m \coloneqq  \begin{cases}
n-1 & \quad \mathfrak{g} = \mathfrak{sl}_n(\C),\\
\lfloor n/2\rfloor & \quad \mathfrak{g} = \mathfrak{so}_n(\C),\\
n & \quad \mathfrak{g} = \mathfrak{sp}_n(\C).
\end{cases}
\end{equation}
\begin{proposition}[Weyl dimension formulae]\label{prop:dim-formula}
Let $n\in \N$ and $\kappa = (\kappa_1,\dots,  \kappa_m) \in \N^m$ with $m$ as in \eqref{eq:m}. An irreducible complex $\g$-module $\W_\kappa $ has dimension
\[
\dim \W_\kappa =
\begin{cases}
\prod\limits_{1 \le i < j \le m} \biggl( \frac{ \sum_{s = i}^{j-1} \kappa_s}{j-i} + 1 \biggr) &\text{if } \mathfrak{g} = \mathfrak{sl}_n(\C), \\
\prod\limits_{1 \le i < j \le m} \biggl( \frac{\sum_{s=i}^{j-1}\kappa_s}{j - i} + 1 \biggr)  \prod\limits_{1 \le i \le j \le m} \biggl(\frac{\sum_{s=i}^{m}\kappa_s + \sum_{t=j}^{m-1}\kappa_t}{2m+1 - i- j} + 1\biggr) &\text{if } \mathfrak{g} = \mathfrak{so}_{2m+1}(\C), \\
\prod\limits_{1 \le i < j \le m} \biggl( \frac{\sum_{s=i}^{j-1}\kappa_s}{j - i} + 1 \biggr) \prod\limits_{1 \le i < j \le m} \biggl( \frac{\sum_{s=i}^{m}\kappa_s + \sum_{t=j}^{m-1}\kappa_t }{2m - i- j} + 1 \biggr) &\text{if }\mathfrak{g} = \mathfrak{so}_{2m}(\C),\\
\prod\limits_{1\le i < j \le n} \biggl( \frac{\sum_{s=i}^{j-1}\kappa_s}{j-i} + 1 \biggr) 
\prod\limits_{1\le i \le j \le n} \biggl( \frac{\sum_{s=i}^n \kappa_s + \sum_{t=  j}^n \kappa_t}{2n+2-i-j} + 1 \biggr) &\text{if } \mathfrak{g} = \mathfrak{sp}_{n}(\C).
\end{cases}
\]
\end{proposition}
For a complex simple Lie group $\G$, an irreducible $\G$-module will mean an irreducible $\mathfrak{g}$-module where $\mathfrak{g}$ is the Lie algebra of $\G$.

\section{Linear representations of $\G$-manifolds}\label{sec:man}

We establish some basic definitions and results about our main object of study, a linear representation of a $\G$-manifold $\M $. The notion has its roots in differential geometry. In differential geometric lingo \cite{Kawa, Lee_2013}, if $\mathcal{N}$ is another $\G$-manifold, then a differentiable map $\varepsilon: \M  \to \mathcal{N}$ is $\G$-equivariant if $\varepsilon(g \cdot x) = g \cdot \varepsilon(x)$ for all $g\in \G$, $x \in \M $.  If, moreover, $\varepsilon$ is an embedding, then $\M $ is called a \emph{$\G$-submanifold} of $\mathcal{N}$. In the special case where the manifold is linear, i.e., $\M  = \V$, a $\G$-module, the $\G$-equivariant embedding $\varepsilon: \G/ \H \to \V$ is called a \emph{Mostow--Palais embedding}.

What we call a linear representation imposes further restrictions on the nature of $\V$ and its $\G$-action, bringing these differential geometric notions closer to group representations:
\begin{definition}[Manifold representation]\label{def:rep}
Let $\G \subseteq  \GL_n(\F)$. A \emph{linear representation} of a $\G$-manifold $\M $ is an equivariant differentiable map   $\varepsilon: \M  \to \V$  where $ \V \subseteq \F^{n \times n}$ and $\G$ acts on $\V$ via matrix multiplication.
\end{definition}
All representations in this article, whether of manifolds, groups, or algebras, will be linear representations, i.e., taking values in $\F^{n \times n}$. Henceforth we will drop the word ``linear'' except in section headings to serve as a reminder.

The requirement that $\G \subseteq  \GL_n(\F)$ in Definition~\ref{def:rep} is strictly speaking unnecessary. For example, for any $\G$ that has an $n$-dimensional group representation $\pi : \G \to \GL_n(\F)$, we may let $\pi(\G)$ play the role of $\G$ above. Nevertheless for convenience we will assume that $\G$ is a Lie subgroup of $\GL_n(\F)$ throughout. The action via matrix multiplication $\G \times \V \to \V$ in the definition could be left or right matrix multiplication, equivalence, congruence, or similarity:
\begin{equation}\label{eq:act}
\begin{gathered}
(A, X) \mapsto A X, \quad (A, X) \mapsto X A^{-1}, \quad  (A_1,A_2, X) \mapsto A_1 X A_2^{-1}, \\
(A, X) \mapsto A X A^\tp, \quad (A, X) \mapsto A X A^{-1},
\end{gathered}
\end{equation}
but we leave open the possibility of yet other unusual actions like bicongruence $(A_1,A_2, X) \mapsto A_1 X A_2^\tp$ (which is the same as equivalence for the groups we consider in this article). The products above are all matrix multiplications so these actions would not be well-defined without the assumptions that $\G \subseteq  \GL_n(\F)$ and $ \V \subseteq \F^{n \times n}$, which will be implicit whenever we speak of a representation of a manifold.

The definition of a representation of $\G$-manifold in Definition~\ref{def:rep} is consistent with both $\G$-equivariant map in differential geometry and  representation of a group in representation theory. In fact, a manifold representation may be regarded as an extension of a group representation to $\G$-manifolds that may not be groups:
\begin{proposition}[Every group representation is a manifold representation]\label{prop:group}
Let  $\pi: \G \to \GL_n(\F) \subseteq \mathbb{F}^{n \times n}$ be a representation of a Lie group $\G$ in the sense of group representation. If we regard $\pi$ as a map $\pi: \G \to \mathbb{F}^{n \times n}$, then it is also a representation of $\G$ in the sense of manifold representation.
\end{proposition}
\begin{proof}
Here $\G$ acts on $\M  \coloneqq \G$ via left multiplication and on $\V \coloneqq \mathbb{F}^{n \times n}$ via $\pi$, i.e., for all $g \in \G$ and $X \in \mathbb{F}^{n \times n}$, $g \cdot X = \pi(g)X$. Clearly $\pi$ is $\G$-equivariant as $\pi(g \cdot x) = \pi(g x) = \pi(g) \pi(x) = g \cdot \pi(x)$ for any $g \in \G$, $x \in \M $.
\end{proof}

The use of group representation lingo in Definition~\ref{def:rep} is suggestive, immediately leading to notions of  equivalence, faithfulness, and irreducibility for representations; minimality for faithful ones; alongside common constructions of new representations from old ones. While these are commonplace in group (and algebra) representations, to the best of our knowledge, they are absent from differential geometry.

\begin{definition}[Equivalence]\label{def:equiv}
Let $\M _1$ and $\M _2$ be $\G$-manifolds. Two representations $\varepsilon_1: \M _1 \to \V_1$ and $\varepsilon_2: \M _2 \to \V_2$ are \emph{equivalent},  denoted $\varepsilon_1 \cong \varepsilon_2$,  if there exist a diffeomorphism $\varphi: \M _1\cong \M _2$ and a $\G$-module isomorphism $\psi: \V_1 \cong \V_2$ such that $\psi \circ \varepsilon_1 = \varepsilon_2 \circ \varphi$ and 
\[
\varphi(g \cdot x) = g \cdot \varphi(x),\quad \psi(g \cdot v) = g \cdot \psi(v),\quad 
\]
for any $g \in \G$, $x \in \M _1$, and $v \in \V_1$.
\end{definition}

The notions of faithfulness and minimality are particularly important for us as a main goal of this article is to classify $\G$-manifolds that have faithful representations and explicitly find all minimal dimensional ones, thereby arriving at sharp effective bounds for their Mostow--Palais embeddings.
\begin{definition}[Faithfulness and minimality]\label{def:faith}
A representation of a $\G$-manifold  $\varepsilon: \M  \to \V$ is \emph{faithful} if it is a $\G$-equivariant embedding.  A faithful representation is \emph{minimal} if there is no faithful representation $\varepsilon': \M  \to \V'$ with $\dim_{\R} \V' < \dim_{\R} \V$.
\end{definition}
Any faithful representation is clearly a Mostow--Palais embedding; it will follow from our classification result that the converse is false. 

The notions of equivalence and faithfulness in Definitions~\ref{def:equiv} and \ref{def:faith} are also consistent with those same notions in group representation theory:
\begin{proposition}
\begin{enumerate}[\normalfont(i)]
  \item If $\pi : \G \to \GL_n(\F) \subseteq \mathbb{F}^{n \times n}$ is faithful in the sense of group representation, then, when regarded as a map $\pi: \G \to \F^{n \times n}$, it is faithful in the sense of manifold representation.
  \item If $\pi_1, \pi_2: \G \to \GL_n(\F) \subseteq \mathbb{F}^{n \times n}$ are equivalent in the sense of group representation, then, when regarded as maps $\pi_1, \pi_2 : \G \to \F^{n \times n}$, they are equivalent in the sense of manifold representation.
\end{enumerate}
\end{proposition}
\begin{proof}
If $\pi$ is a faithful Lie group representation,  then it is a $\G$-equivariant embedding of $\G$ into $\GL_n(\F)$. If $\pi_1$ and $\pi_2$ are equivalent Lie group representations, then there exists $A \in \GL_n(\F)$ such that for all $g \in \G$, $\pi_2(g) = A \pi_1(g) A^{-1}$. We verify that $\pi_1 \cong \pi_2$ by taking $\varphi = \operatorname{id}$, the identity map on $\G$, and 
\[
  \psi: \F^{n \times n} \to \F^{n \times n}, \quad X \mapsto AXA^{-1}
\]
in Definition~\ref{def:equiv}. For all $g \in \G$ and $X \in \F^{n \times n}$, 
\[
  \pi_2(g)\psi(X) = (A\pi_1(g)A^{-1})(AXA^{-1}) = A\pi_1(g)XA^{-1} = \psi(\pi_1(g)X)
\]
and $\psi \circ \pi_1(g) = A\pi_1(g)A^{-1} = \pi_2(g) \circ \varphi$.
\end{proof}

Standard constructions of new group representations from old ones extend to manifold representations almost verbatim, as Definitions~\ref{def:sum+prod} and \ref{def:restricted} show.
\begin{definition}[Direct sum and tensor product]\label{def:sum+prod}
Let $\varepsilon_1: \M  \to \V_1$ and $\varepsilon_2: \M  \to \V_2$ be two representations of $\G$-manifold $\M $.  The \emph{direct sum} $\varepsilon_1 \oplus \varepsilon_2: \M  \to \V_1 \oplus \V_2$ is defined by $\varepsilon_1 \oplus \varepsilon_2(x) = \varepsilon_1(x) \oplus \varepsilon_2(x)$. The \emph{tensor product} $\varepsilon_1 \otimes \varepsilon_2: \M  \to \V_1 \otimes \V_2$ is defined by $\varepsilon_1 \otimes \varepsilon_2(x) = \varepsilon_1(x) \otimes \varepsilon_2(x)$.
\end{definition}
The equivalence of representations is not necessarily preserved by direct sums (see Example~\ref{ex:inter}) or tensor products.

Given a manifold representation $\varepsilon: \M  \to \V$ of a $\G$-manifold, we may limit it in three ways: to $\mathcal{N} \subseteq \M $, to $\H \subseteq \G$, or to $\V' \subseteq \V$; each gives a meaningful manifold representation.
\begin{definition}[Restricted, restrained, sub-, and complementary representations]\label{def:restricted}
Let $\varepsilon: \M  \to \V$ be a representation.   If $\mathcal{N} \subseteq \M $ is a $\G$-submanifold, then $\varepsilon\vert_{\mathcal{N}}$ is the \emph{restricted representation} of $\mathcal{N}$. If $\H \subseteq \G$ is a closed subgroup, then $\varepsilon$ is \emph{restrained representation} of $\M $ viewed as a $\H$-manifold.  If $\V' \subseteq \V$ is a $\G$-submodule, then $\varepsilon' = p' \circ \varepsilon$ is the \emph{subrepresentation} of $\M $ in $\V'$. If $\V'$ has a complementary subspace $\V''$, i.e., $\V = \V' \oplus \V''$, that is also a $\G$-submodule, then $\varepsilon'' = p'' \circ \varepsilon$ is the \emph{complementary representation} of $\varepsilon'$. Here $p'$ and $p''$ denote the orthogonal projections onto $\V'$ and $\V''$ respectively.
\end{definition}

The notions of irreducibility and indecomposability in group representations readily extend to manifold representations.
\begin{definition}[Irreducibility and indecomposability]\label{def:irred}
A representation $\varepsilon: \M  \to \V$ is \emph{irreducible} if $\V$ is irreducible as a $\G$-module; it is \emph{indecomposable} if $\varepsilon = \varepsilon_1 \oplus \varepsilon_2$ implies that  either $\varepsilon_1$ or $\varepsilon_2$ is the trivial representation $\M  \to \{0\}$. 
\end{definition}

Decompositions into indecomposables for arbitrary $\G$ and irreducibles for reductive $\G$ likewise extend to manifold representations.
\begin{proposition}[Complete decomposability and reducibility]\label{prop:decomp-irrep}
Every representation of a $\G$-manifold $\varepsilon: \M  \to \V$ can be written as $\varepsilon = \varepsilon_1 \oplus \dots \oplus \varepsilon_k$,  where each $\varepsilon_j: \M  \to \V_j$ is an indecomposable representation of $\M $ and $\V = \V_1 \oplus \dots \oplus \V_k$ is a unique decomposition up to isomorphism. If, furthermore, $\G \subseteq \GL_n(\F)$ is reductive, then $\varepsilon_1,\dots,\varepsilon_k$ can be chosen to be irreducible.
\end{proposition}
\begin{proof}
By the Krull--Schmidt theorem \cite[Theorem~12.9]{AF1992}, the finite-dimensional $\G$-module $\V$ admits an indecomposable decomposition $\V = \V_1 \oplus \dots \oplus \V_k$ that is unique up to isomorphism. We set $\varepsilon_j \coloneqq p_j \circ \varepsilon$ where $p_j: \V \to \V_j$ is the projection, $ j =1,\dots, k$. Suppose $\varepsilon_j = \varepsilon_j' \oplus \varepsilon_j''$ where $\varepsilon_j': \M  \to \V_j'$ and $\varepsilon_j'': \M  \to \V_j''$ with $\V_j = \V_j' \oplus \V_j''$. As $\V_j$ is indecomposable, either $\V_j'$ or $\V_j'' = \{0\}$, i.e., $\varepsilon_j'$ or $\varepsilon_j''$ is trivial. So $\varepsilon_j$ is indecomposable. For linearly reductive $\G$, the decomposition of $\V$ may be chosen to be irreducible \cite[Section~7.3]{Procesi} and in which case $\varepsilon_j$ is irreducible.
\end{proof}

A representation of a $\G$-manifold may be decomposed in yet another sense, namely, via the $\G$-orbits in $\M $. For any $x \in \M $, we denote its orbit and stabilizer under the action of $\G$ by
\[
\G\cdot x \coloneqq \{g\cdot x\in \M : g\in \G\},\quad \G_x \coloneqq \{g \in \G: g\cdot x = x\}.
\]
The following corollary of Proposition~\ref{prop:decomp-irrep} shows that a representation of a $\G$-manifold is completely determined by its restrictions to orbits.

\begin{corollary}[Characterization by orbits]\label{cor:orbits}
Let $\M $ be a $\G$-manifold and $\varepsilon: \M  \to \V$ any nonzero differentiable map.  Then $\varepsilon$ is a representation of $\M $ if and only if for every $x \in \M $, the restriction $\varepsilon\vert_{\G \cdot x}: \G \cdot x \to \V$ is a representation of the orbit $\G \cdot x$. Furthermore, $\varepsilon$ is irreducible if and only if $\varepsilon\vert_{\G \cdot x}$ is irreducible for every $x \in \M $.
\end{corollary}

When discussing a representation $\varepsilon: \M  \to \V$, there are two $\G$-actions involved, one on the $\G$-manifold $\M $, and another on the $\G$-module $\V$. To avoid ambiguity, and since the $\G$-action on $\M $ is transitive, in the remainder of this article we will assume $\M  = \G/\H$ with $\H$ a closed subgroup of $\G$.  Henceforth, whenever we speak of a $\G$-action, $\G$-orbit,  or $\G$-stabilizer, it will exclusively refer that of the $\G$-module $\V$.

\section{Linear representations of homogeneous spaces}\label{sec:homo}

By the discussion at the end of the last section, we limit our attention to $\G$-manifolds that are homogeneous spaces of the form $\G/\H$.  Corollary~\ref{cor:orbits} ensures that representations respect this quotient structure. It turns out that we may  further narrow our focus to the following kind of homogeneous spaces:
\begin{definition}[Irreducibly faithfully representable spaces]
  We say that $\G/\H$ is \emph{irreducibly faithfully representable} if, for any closed subgroups $\H_1$, $\H_2$ with
\[
\H \subseteq \H_1 \subseteq \G, \quad \H \subseteq \H_2 \subseteq \G,
\]
and $\G/\H_1$, $\G/\H_2$ both admit faithful representations, then $\H = \H_1 \cap \H_2$ implies that $\H = \H_1$ or $\H = \H_2$.
\end{definition}
Irreducibly faithfully representable homogeneous spaces admit an alternative characterization.

\begin{proposition}[Irreducibly faithfully representability]
A homogeneous space $\G/\H$ is irreducibly faithfully representable if and only if for any faithful representation of the form $\varepsilon = (\varepsilon_1,\dots,  \varepsilon_k): \G/\H \to \V = \V_1 \oplus \dots \oplus \V_k$,  where $\varepsilon_1,\dots,  \varepsilon_k$ are representations of $\G/\H$,  there exists a faithful $\varepsilon_j$ for some $j \in \{1,\dots k\}$. 
\end{proposition}
\begin{proof}
Let $\G/\H$ be irreducibly faithfully representable.  If $\varepsilon$ is faithful, then there are closed subgroups $\H_1,\dots,  \H_k$ containing $\H$,  such that each $\varepsilon_j: \G/\H \to \V_j$ factors through some faithful representation $\eta_j: \G/\H_j \to \V_j$,  and $\H =  \H_1 \cap \dots \cap \H_k$.  By definition,  $\varepsilon_j$ must be faithful for some $j \in \{1,\dots, k\}$.  Conversely,  let $\H_1,\H_2$ be two closed subgroups of $\G$ such that $\H = \H_1 \cap \H_2$ and that $\G/\H_i$ admits a faithful representation $\varepsilon_i: \G/\H_i \to \V_i$, $i = 1,2$.
These induce a faithful representation $\varepsilon = (\varepsilon_1,\varepsilon_2): \G/\H = \G/(\H_1 \cap \H_2) \to \V_1 \oplus \V_2$. By assumption,  either $\varepsilon_1$ or $\varepsilon_2$ must be faithful,  so either $\H = \H_1$ or $\H = \H_2$.
\end{proof}

\subsection{Mostow--Palais embeddings of homogeneous spaces}\label{sec:orb}

While we are ultimately interested in faithful representations,  the results in this section hold true more generally for Mostow--Palais embeddings and we will state them as such. In other words, here we study $\G$-equivariant embeddings $\varepsilon: \G/ \H \to \V$ with no further assumption on $\V$ other than that it is a $\G$-module.  Our first observation is that any Mostow--Palais embedding of $\G/ \H$ in $\V$ must be a $\G$-orbit:
\begin{lemma}\label{lem:embedding}
There exists a Mostow--Palais embedding $\varepsilon:  \G/ \H \to \V$ if and only if there exists $v \in \V$ such that $\H = \G_v$.  In this case, we must have $\varepsilon( \H ) = v$.
\end{lemma}
We next show that Mostow--Palais embeddings respect direct sums.
\begin{lemma}\label{lem:min-sum}
Let $\V_1$ and $\V_2$ be $\G$-modules.  A map $\varepsilon: \G/\H \to \V_1 \oplus \V_2$ is a Mostow--Palais embedding if and only if $\varepsilon = \overline{\varepsilon}_1 \oplus \overline{\varepsilon}_2$ for some Mostow--Palais embeddings $\varepsilon_i: \G/\H_i \to \V_i$, $i =1,2$,  where $\H_1,\H_2 \subseteq \G$ are closed subgroups such that $\H_1 \cap \H_2 = \H$, and $\overline{\varepsilon}_i: \G/\H \to \V_i$ is $\varepsilon_i$ composed with the quotient map $ \G/\H \to  \G/\H_i$, $i = 1,2$. 
\end{lemma}
\begin{proof}
Let $(v_1,v_2) \in \V_1 \oplus \V_2$ and $g \in \G$. Then
\begin{equation}\label{eq:orb-cap}
\G_{(v_1,v_2)} = \G_{v_1} \cap \G_{v_2}.
\end{equation}
Let $\varepsilon = \overline{\varepsilon}_1 \oplus \overline{\varepsilon}_2$ as in the statement of the lemma.  By Lemma~\ref{lem:embedding}, there exist $v_1 \in \V_1$ and $v_2 \in \V_2$ such that $\H_1 = \G_{v_1}$ and $\H_2 = \G_{v_2}$.  So $\H = \G_{(v_1,v_2)}$ by \eqref{eq:orb-cap}. Applying Lemma~\ref{lem:embedding} again,  we conclude that $\varepsilon$ is a Mostow--Palais embedding as $\varepsilon(\H) = (\varepsilon_1(\H_1), \varepsilon_2(\H_2)) = (v_1,v_2)$.

For the converse, let $\varepsilon: \G/\H \to \V_1 \oplus \V_2$ be a Mostow--Palais embedding.  By Lemma~\ref{lem:embedding}, there exist $v_1 \in \V_1$ and $v_2 \in \V_2$ such that $\H = \G_{(v_1,v_2)}$. Let $\H_1 = \G_{v_1}$ and $\H_2 = \G_{v_2}$.  Then $\H = \H_1 \cap \H_2$ by \eqref{eq:orb-cap}. Applying Lemma~\ref{lem:embedding} again,  we obtain Mostow--Palais embeddings $\varepsilon_i: \G/\H_i \to \V_i$, $i =1,2$. By construction,  we have $\varepsilon(\H) = (v_1,v_2) = (\varepsilon_1(\H_1),\varepsilon_2(\H_2))$. Hence $\varepsilon = \overline{\varepsilon}_1 \oplus \overline{\varepsilon}_2$. 
\end{proof} 

By repeatedly applying Lemma~\ref{lem:min-sum},  we obtain the following.
\begin{proposition}\label{prop:min-sum}
Let $\G \subseteq \GL_n(\F)$ be a linearly reductive Lie group and let $\V$ be a $\G$-module.  A map $\varepsilon: \G/\H \to \V$ is a Mostow--Palais embedding if and only if $\varepsilon = \overline{\varepsilon}_1 \oplus \dots \oplus \overline{\varepsilon}_k$ where
\begin{enumerate}[\upshape (i)]
\item $\H_1,\dots,\H_k$ are closed subgroups of $\G$ and $\H = \H_1 \cap \dots \cap \H_k $; 

\item $\V_1,\dots,\V_k$ are irreducible $\G$-modules and $\V = \V_1 \oplus \dots \oplus \V_k$;

\item $\varepsilon_i : \G/\H_i \to \V_i$ is a Mostow--Palais embedding, $i=1,\dots, k$;

\item$\overline{\varepsilon}_i: \G/\H \to \V_i$ is $\varepsilon_i$ composed with the quotient map $ \G/\H \to  \G/\H_i$, $i = 1,\dots,k$.
\end{enumerate}
\end{proposition}

We highlight a potential pitfall here: Suppose $\varepsilon_i: \G/\H_i \to \V_i$ and $\varepsilon'_i: \G/\H'_i \to \V'_i$ are Mostow--Palais embeddings with $\varepsilon_i \cong \varepsilon'_i$, $i = 1,2$. It is possible that $\overline{\varepsilon}_1 \oplus \overline{\varepsilon}_2 \not\cong \overline{\varepsilon}_1' \oplus \overline{\varepsilon}_2'$. In other words,  for $(u_1,u_2), (v_1,v_2) \in \V_1 \oplus \V_2$ satisfying
\begin{equation}\label{eq:stab}
 \G_{u_1} \cong \G_{v_1},\quad \G_{u_2} \cong \G_{v_2}, 
\end{equation}
it is not necessarily true that $\G_{(u_1,u_2)} \cong \G_{(v_1,v_2)}$. We provide an example below
\begin{example}\label{ex:inter}
Let $\G = \SO_n(\R)$ act on $\U_1 = \U_2 = \R^n$ by left multiplication.  Let $\{e_1,\dots,  e_n\}$ be the standard basis of $\R^n$.  Then
\[
\G_{e_1} = \biggl\{\begin{bmatrix}
1 \\ & Q 
\end{bmatrix}: Q \in \SO_{n-1}(\R)  \biggr\} \cong \SO_{n-1}(\R) \cong \G_{e_n} = \biggl\{\begin{bmatrix}
  Q \\ &1
  \end{bmatrix}: Q \in \SO_{n-1}(\R)  \biggr\}.
\]
So  $(e_1,  e_1)$ and $(e_1,  e_n)$ satisfy \eqref{eq:stab}.  However,  a direct calculation shows that
\[
\G_{(e_1,e_1)} = \biggl\{\begin{bmatrix}
  1 \\ & Q
  \end{bmatrix}: Q \in \SO_{n-1}(\R)  \biggr\} \cong \SO_{n-1}(\R) 
  \not\cong 
  \left\lbrace \begin{bmatrix}
    1 \\ & U \\ &&1
    \end{bmatrix}: U \in \SO_{n-2}(\R)  \right\rbrace = \G_{(e_1,e_n)}.
\]
\end{example}
In Section~\ref{sec:orbit}, we will see that \eqref{eq:stab} is equivalent to $\G_{v_1} = A \G_{u_1} A^{-1}$, $\G_{v_2} = C \G_{u_2} C^{-1}$ for some $A, C \in \GL_n(\F)$. In which case, since
\[
\G_{(u_1,u_2)} = \G_{u_1} \cap \G_{u_2},\quad \G_{(v_1,v_2)} = A \G_{u_1} A^{-1} \cap C \G_{u_2} C^{-1},
\]
there is no guarantee that $\G_{(u_1,u_2)} \cong \G_{(v_1,v_2)}$ unless $A = C$.

Incidentally, Example~\ref{ex:inter} also shows that the homogeneous space $\SO_n(\R)/\SO_{n-2}(\R)$ is not irreducibly faithfully representable and that equivalence of representations is not preserved under direct sum. 

Lastly, we show that if $\G_\R$ is a real form of a complex $\G$, then Mostow--Palais embeddings of $\G_\R$-homogeneous spaces into complex $\G_\R$-modules are determined by those of $\G$-homogeneous spaces into $\G$-modules, important for our classification in Section~\ref{sec:red}.

\begin{lemma}[Mostow--Palais embeddings of real forms]\label{lem:comp-real}
Let $\G_\R$ be a real form of a simple complex Lie group $\G$. If $\varepsilon: \G_\R/\H \to \W$ is a Mostow--Palais embedding into a complex $\G_\R$-module $\W$, then $\W$ is a $\G$-module and there exist a closed subgroup $\H' \subseteq \G$ such that $\H = \H' \cap \G_\R$ and a Mostow--Palais embedding $\varepsilon': \G/\H' \to \W$ such that $\varepsilon= \varepsilon' \circ \iota$ where $\iota: \G_\R/\H \to \G/\H'$,  $g \H \mapsto g\H'$, is the natural inclusion.  Moreover, if $\W = \U \otimes_{\R} \C$ for some real $\G_\R$-module $\U$,  then there exist $v_1, v_2 \in \U$ such that $\H = (\G_{\R})_{v_1 + i v_2} = (\G_{\R})_{v_1} \cap (\G_{\R})_{v_2}$. 
\end{lemma}
\begin{proof}
Since $\G_\R$ is simple, $\W$ is a completely reducible $\G_\R$-module. As irreducible complex $\G_\R$-modules are exactly irreducible $\G$-modules, $\W$, when viewed as a $\G$-module, is a direct sum of complex irreducible $\G$-modules. By Lemma~\ref{lem:embedding}, there exists $v \in \W$ such that $\H = (\G_\R)_v$ and $\varepsilon(g\H) = g\cdot v$. Let $\H' \coloneqq \G_v$. Then $\H = \H' \cap \G_\R$. Applying Lemma~\ref{lem:embedding} to the action of $\G$ on $v$ yields
\[
\varepsilon': \G/\H' \to \W \quad  g\H' \mapsto g \cdot v
\]
as a Mostow--Palais embedding. It remains check that for any $g\H \in \G_\R/\H$, $\varepsilon'\bigl(\iota(g\H)\bigr) = \varepsilon'(g\H') = g\cdot v = \varepsilon(g\H)$. 
The case of $\W = \U \otimes_\R \C$ follows from Lemma~\ref{lem:embedding} and Proposition~\ref{prop:min-sum}.
\end{proof}

\subsection{Stabilizers of $\G$-actions}\label{sec:orbit}

The homogeneous spaces of interest to us would all take the form $\G/\H$ where $\G \subseteq \GL_n(\F)$ is a Lie subgroup and $\H$ is a stabilizer of one of three standard actions: equivalence, congruence, or similarity. For any $\GL_n(\mathbb{F})$-module $\V \subseteq \F^{n \times n}$ and $X \in \V$, we have
\begin{equation}\label{eq:stabilizer}
  \G_X = \{g \in \G : g \cdot X = X\} = \GL_n(\F)_X \cap \G.
\end{equation}
We next show that the set of homogeneous spaces $\G/\H$ that has a faithful representation is in one-to-one correspondence with the set of possible stabilizers $\G_X$, $X \in \V \subseteq \F^{n \times n}$, again important for our classification in Section~\ref{sec:red}.
\begin{lemma}\label{lem:bijection}
Define an equivalence relation $\H_1 \sim \H_2$ if $\G/\H_1 \cong \G/\H_2$ and set
\[
\mathcal{E} (\G) \coloneqq \left\lbrace
\H \subseteq \G: \H \text{ closed in } \G \text{ and } \G/\H \text{ is a }\G\text{-submanifold of }\V \right\rbrace/\sim.
\]
Define another equivalence relation $v_1 \sim v_2$ if $\G\cdot v_1 \cong \G\cdot v_2$ and set  $\mathcal{F} (\V) \coloneqq \V /\sim$.  Then for any Mostow--Palais embedding $\varepsilon : \G/ \H \to \V$, the map 
\[
\varphi: \mathcal{E}(\G) \to \mathcal{F} (\V),\quad \varphi([\H]) =  [\varepsilon(\H)]
\]
is well-defined and bijective. In particular, $\varphi$ is independent of the choice of $\varepsilon$. 
\end{lemma}
\begin{proof}
Suppose $\varepsilon'$ is another Mostow--Palais embedding of $\G/\H$ into $\V$. Let $v \coloneqq \varepsilon(\H) $ and $v' \coloneqq \varepsilon'(\H) $.  Then $\G_{v} = \G_{v'} = \H$ and so $\G\cdot v \cong \G \cdot v' \cong \G/\H$.  Thus $\varphi$ does not depend on the choice of $\varepsilon$.  
The surjectivity of $\varphi$ follows from Lemma~\ref{lem:embedding}. For injectivity, suppose $\varphi([\H_1]) = \varphi([\H_2]) = [v]$.  Then there are Mostow--Palais embeddings $\varepsilon_1: \G/\H_1 \to \V$ and $\varepsilon_2: \G/\H_2 \to \V$ such that
\[
\G/\H_1 \cong \G\cdot \varepsilon_1(\H_1) \cong \G\cdot v \cong  \G \cdot \varepsilon_2(\H_2) \cong \G/\H_2,
\]
implying $[\H_1] = [\H_2]$.
\end{proof}

We will reduce our classification tasks in Sections~\ref{sec:A}--\ref{sec:C} to the case of simple groups. Thus equivalence reduces to left and right multiplications; and as the latter is dual to the former, it suffices to determine the stabilizer for left multiplication. So for our purpose, it suffices to determine the stabilizer $\GL_n(\F)_X$ under left multiplication, congruence and similarity. Surprisingly, they cannot be easily found in standard references; we will describe these stabilizers below for our later use. 

Let $\G_1 \subseteq \GL_p(\F)$ and $\G_2 \subseteq \GL_{n-p}(\F)$ be closed subgroups. We define the subgroup 
\begin{equation}\label{eq:semidirect}
\P(\G_1,  \G_2) \coloneqq \left\lbrace
\begin{bmatrix}
    A_1 & C \\
    0 & A_2
\end{bmatrix} \in \GL_n(\F): A_1 \in \G_1, \; A_2 \in \G_2,\; C \in \F^{p \times (n-p)}
\right\rbrace.
\end{equation}
For a trivial group  $G_1 = \{I_r\}$, we will write $\P(I_r,  \G_2)$, omitting the braces.

\begin{lemma}[Stabilizers of left multiplication]\label{lem:rectangle}
Let $\GL_n(\F)$ act on $\F^{n \times k}$, $k \le n$, by left multiplication
\[
\GL_n(\F) \times \F^{n\times k} \to \F^{n\times k},\quad (A,X) \mapsto AX.
\]
For any $X \in \F^{n \times k}$ with $\rank (X) = r$,  the stabilizer $\GL_n(\F)_X = Q \P\bigl(I_r, \GL_{n-r}(\F)\bigr) Q^{-1}$ for some $Q \in \O_n(\mathbb{R})$ if $\F = \R$ or $Q \in \Un_n$ if $\F = \C$.
\end{lemma}
\begin{proof}
Over $\R$, by the QR decomposition, there exist $Q \in \O_n(\R)$,  $S \in \R^{r \times (k-r)}$,  a permutation matrix $\Pi \in \R^{k \times k}$, and an upper triangular matrix $R \in \GL_r(\R)$,  such that 
\[
X = Q \begin{bmatrix}
    R & S \\ 0 &0
\end{bmatrix} \Pi^\tp.
\] 
Given $A \in \GL_n(\R)$,  we partition $Q^{-1} A Q \eqqcolon C = \begin{bsmallmatrix}
    C_{11} &C_{12}\\
    C_{21} &C_{22}
\end{bsmallmatrix}$ into blocks $C_{11} \in \R^{r \times r}$, $C_{12} \in \R^{r \times (n-r)}$, $C_{21} \in \R^{(n-r) \times r}$, and $C_{22} \in \R^{(n-r) \times (n-r)}$, from which it is clear that  $A \in \GL_n(\R)_X$ if and only if $C_{11} = I_r$ and $C_{21} = 0$. The same argument holds over $\C$ with $\Un_n$ in place of $\O_n(\R)$.
\end{proof}

Congruence defines an action on both the space of skew-symmetric matrices $\Alt^2(\F^n)$ and that of symmetric matrices $\Sym^2(\F^n)$. We characterize the respective stabilizers in the next two lemmas.
\begin{lemma}[Skew-symmetric stabilizers of congruence]\label{lem:antisymmetric}
Let $\GL_n(\F)$ act on $\Alt^2(\F^n)$ by congruence: 
\[
\GL_n(\F) \times \Alt^2(\F^n) \to \Alt^2(\F^n),\quad (A,  X) \mapsto A X A^\tp.
\]
For any $X \in \Alt^2(\F^n)$ with $\rank(X) = 2r$,  the stabilizer $\GL_n(\F)_X = Q \P\bigl(\Sp_{2r}(\F;\Omega),  \GL_{n-2r}(\F)\bigr) Q^{-1}$ for some $\Omega \in \Alt^2(\F^{2r}) \cap \GL_{2r}(\F)$, and $Q \in \O_n(\mathbb{R})$ if $\F = \R$ or $Q \in \Un_n$ if $\F = \C$.
\end{lemma}
\begin{proof}
For a positive integer $m$, write
\[
\Omega_{2m} \coloneqq \diag \biggl( 
\underbrace{\begin{bmatrix}
0 & 1 \\
-1 & 0 
\end{bmatrix},\dots,  \begin{bmatrix}
0 & 1 \\
-1 & 0 
\end{bmatrix}}_{m \text{~times}} \biggr) \in \Alt^2(\F^{2m}). 
\] 
By \cite{Thompson1988} and \cite{Youla}, $X$ has the canonical form
\[
X = Q \diag (\Omega, 0_{n-r}) Q^\tp, \quad \Omega =  \diag(\lambda_1 \Omega_{2n_1}, \dots,  \lambda_s \Omega_{2n_s}),
\]
with $\lambda_1, \dots, \lambda_s \in \R \setminus \{0\}$ and $(n_1,\dots, n_s) \in \mathbb{N}^s$ satisfying $2n_1 + \dots + 2n_s = 2r$.  Given $A \in \GL_n(\F)$, we partition $Q^{-1} A Q \eqqcolon C = \begin{bsmallmatrix}
    C_{11} &C_{12}\\
    C_{21} &C_{22}
\end{bsmallmatrix}$ into blocks $C_{11} \in \R^{2r \times 2r}$, $C_{12} \in \R^{2r \times (n-2r)}$, $C_{21} \in \R^{(n-2r) \times 2r}$, and $C_{22} \in \R^{(n-2r) \times (n-2r)}$, from which it is clear that $A \in \GL_n(\F)_X$ if and only if $C_{11} \Omega C_{11}^\tp = \Omega$ and $C_{21} = 0$. 
\end{proof}

\begin{lemma}[Symmetric stabilizers of congruence]\label{lem:symmetric}
Let $\GL_n(\F)$ act on $\Sym^2(\F^n)$ by congruence: 
\[
\GL_n(\F) \times \Sym^2(\F^n) \to \Sym^2(\F^n),\quad (A,  X) \mapsto A X A^\tp.
\]
For any $X \in \Sym^2(\F^n)$ with $\rank(X) = r$,  the stabilizer $\GL_n(\F)_X = Q \P\bigl(\O_r(\F, B),  \GL_{n-r}(\F)\bigr)Q^{-1}$ for some  $B \in \Sym^2(\F^r) \cap \GL_r(\F)$, and $Q \in \O_n(\mathbb{R})$ if  $\F = \R$ or $Q \in \Un_n$ if $\F = \C$.
\end{lemma}
\begin{proof}
Over $\R$, by the spectral decomposition for real symmetric matrices,
\[
X = Q \diag (\lambda_1 I_{n_1},  \dots, \lambda_m  I_m,  0_{n-r})Q^\tp,
\]
where $Q \in \O_n(\R)$, $(\lambda_1,\dots,  \lambda_m) \in (\R\setminus\{0\})^m$, and $(n_1,\dots,  n_m) \in \N^m$ satisfying $n_1 + \dots + n_m = r$. The required characterization of $\GL_n(\R)_X$ follows from the same argument in the proof of Lemma~\ref{lem:antisymmetric}. Over $\C$, the Takagi decomposition for complex symmetric matrices \cite{Takagi1924} yields the required result with $Q\in \Un_n$.
\end{proof}

For positive integers $k,l, k_1,\dots,  k_p$,  we let $\UT(k,l,\F)$ denote the space of $k \times l$ upper triangular Toeplitz matrices over $\F$, and $\BUT(2k,2l,\F)$ the space of $k \times l$ block upper triangular Toeplitz matrices over $\F$ with blocks of the form $\bigl[\begin{smallmatrix}
    a & b \\
    -b & a
    \end{smallmatrix}\bigr]$ for some $a,b\in \R$. In addition, we define
\begin{align*}
\UT_{k_1,\dots,  k_p}(\F) &\coloneqq \bigl\lbrace
(T_{ij})_{i,j=1}^p \in \GL_{k_1 + \dots + k_p}(\F): T_{ij} \in \UT(k_i,  k_j,\F),\;  i,j = 1,\dots,  p
\bigr\rbrace,  \\
\BUT_{2k_1,\dots, 2k_p}(\F) &\coloneqq \bigl\{(T_{ij})_{i,j=1}^p \in \GL_{2k_1 + \dots + 2k_p}(\F): T_{ij} \in \BUT(2k_i,2k_j,\F),\; i,j = 1,\dots, p \bigr\},
\end{align*}
and we drop the parenthetical $\F$ when the choice of field is clear. 
It is easy to see that both $\UT_{k_1,\dots,  k_p}$ and $\BUT_{2k_1,\dots, 2k_p}$ are Lie subgroups.  The next lemma follows immediately from \cite[Theorems~9.1.1 and 12.4.2]{Gohberg_Lancaster_Rodman}.
\begin{lemma}[Stabilizers of similarity]\label{lem:JCF}
Let $\GL_n(\F)$ act on $\F^{n \times n}$ by similarity: 
\[
\GL_n(\F) \times \F^{n\times n} \to \F^{n\times n},\quad (A,X) \mapsto A X A^{-1}.
\]
For any $X \in \F^{n \times n}$, the stabilizer
\[
\GL_n(\F)_X = A \Bigl( \prod_{i=1}^p \UT_{n_{i,1},\dots,  n_{i,s_i}} \times  \prod_{j=1}^q \BUT_{m_{j,1},\dots,m_{j,t_j}} \Bigr) A^{-1}, 
\]
for some $A \in \GL_n(\F)$, positive integers $n_{i,1},\dots, n_{i,s_i}$, $i =1,\dots, p$, and $m_{j,1},\dots,  m_{j,t_j}$, $j =1,\dots, q$, satisfying
\[
n = \sum_{i=1}^p (n_{i,1} + \dots + n_{i,s_i}) + \sum_{j=1}^q (m_{j,1} + \dots + m_{j,t_j}).
\]
Moreover,  when $\F = \C$,  we have $m_{j,1} = \dots = m_{j,t_j} = 0$ for all $j =1,\dots, q$.
\end{lemma}

\section{Reduction of the main classification problem}\label{sec:red}

In this brief section, we outline the general arguments used to obtain our classification of $\G$-manifolds that have a faithful representation, i.e., what we promised in item~\ref{it:class} of Section~\ref{sec:new}. We consider all $\G$-manifolds that have transitive $\G$-actions by
\begin{equation}\label{eq:G}
\G \in \{\SL_n(\F), \SO_n(\F),\Sp_n(\F), \SU_n, \SO_{p,q}, \Sp_n: \F = \R \text{ or } \C\}.
\end{equation}
The list contains simple classical groups of types A, B, C, D including all complex groups, all real forms that are split or compact, and $\SO_{p,q}$, which is neither split nor compact when $|p - q| \ge 2$ and $p,q > 0$. The only omissions are four esoteric cases that are also neither compact nor split, namely, $\SO^*_{2n}$, $\SU_{p, q}$, $\Sp_{p, q}$, and $\SU^*_{2n}$, as well as the exceptional groups, which we will leave to future work.

The main effort for this classification will be carried out over the next three sections on a case-by-case basis: Section~\ref{sec:A} for classical groups of type A, Section~\ref{sec:BD} for types B or D,  and Section~\ref{sec:C} for type C. From a high-level perspective, each section follows the same approach below.
\begin{description}
\item[Complex representations of $\G/\H$ for complex groups $\G$] By Proposition~\ref{prop:min-sum}, to classify faithful representations $\varepsilon : \G/\H \to \W \subseteq \C^{n \times n}$, it suffices to consider only irreducible ones, which can be done by determining the stabilizer $\G_X$ for $X \in \W$ by Lemma~\ref{lem:bijection}. In Sections~\ref{sec:slnc},~\ref{sec:sonc}, and~\ref{lem:spnc}, we first list all possible irreducible complex $\G$-modules, then determine all possible $\G_X$'s using results from Section~\ref{sec:orbit}, obtaining a classification.
\item[Real representations of $\G_\R/\H$ for real forms $\G_\R$] By Lemma~\ref{lem:comp-real}, it suffices to classify faithful representations $\varepsilon : \G_\R/\H \to \U \subseteq \R^{n \times n}$ where $\U$ is a real $\G_\R$-module. Any such $\U$ must take the form $\W = \U \otimes_{\R} \C$ for some complex $\G$-module $\W$ subjected to specific constraints depending on both the real form and the type. In Sections~\ref{sec:slnr},~\ref{sec:sun},~\ref{sec:sopq},~\ref{sec:spnr},~\ref{sec:spn}, we first list all possible irreducible real $\G$-modules $\U$, then determine all possible $\G_X$'s using results from Section~\ref{sec:orbit}, obtaining a classification.
\end{description}
We emphasize that the devil is in the details --- the low-level workings in Sections~\ref{sec:A}--\ref{sec:C} differ from one type to the next and cannot be handwaved away. But we do our best to exploit similarities and avoid repetition in Sections~\ref{sec:A}--\ref{sec:C}.

We also need the following observation, which ought to be standard but apparently is not; so we record it below. Recall from Section~\ref{sec:rep} that irreducible $\G$-modules $\W_\kappa$ are indexed by $\kappa \coloneqq (\kappa_1,\dots,  \kappa_m) \in \N^m$ with $m$ as in \eqref{eq:m} \cite[Chapter~3]{Goodman_Wallach}. The value of $\dim\W_{\kappa}$ is given by the Weyl dimension formulae in Proposition~\ref{prop:dim-formula}.
\begin{lemma}[Dimension inequality for irreducible $\G$-modules]\label{lem:dimcut}
Let $\G$ be $\SL_n(\C)$, $\SO_n(\C)$, or $\Sp_n(\C)$. If $\W_\kappa$ is the irreducible $\G$-module indexed by $\kappa = (\kappa_1,\dots,\kappa_s,0,\dots,0) \in \N^m$, then $\dim \W_{\kappa - e_i} < \dim \W_{\kappa}$ for any $i = 1, \dots, s$.  Here $e_i$ is the $i$th standard basis vector in $\F^m$.
\end{lemma}

\section{Faithful linear representations of Type~$\mathrm{A}$ spaces}\label{sec:A}

In this section we will find all homogeneous spaces $\G/\H$ that have faithful representations, where $\G$ is either $\SL_n(\C)$ or its real forms $\SL_n(\R)$ and $\SU_n$. 

\subsection{Faithful linear representations of $\SL_n(\C)$-spaces}\label{sec:slnc}

Let $\W_{\kappa}$ be a complex irreducible $\SL_n(\C)$-module where $\kappa = (\kappa_1,\dots,  \kappa_{n-1})\in \N^{n-1}$.  By Proposition~\ref{prop:dim-formula}, we immediately have 
\begin{equation}\label{eq:dimSym}
\dim \W_{\kappa_{n-1},\dots,\kappa_{1}} = \dim \W_{\kappa_1,\dots,  \kappa_{n-1}}.
\end{equation}
\begin{lemma}[Low dimension $\SL_n(\C)$-module]\label{lem:slnc}
Let $n\ge 9$.  A complex irreducible $\SL_n(\C)$-module $\W$ has dimension at most $n^2$ if and only if it is isomorphic to $\C$,  $\C^n$,  $\Alt^2(\C^n)$,  $ \Sym^2(\C^n)$ or $\sl_n(\C)$.
\end{lemma}
\begin{proof}
Since $\W \cong \W_{\kappa}$ for some $\kappa \in \N^{n-1}$,  it suffices to prove $\dim \W_{\kappa} \le n^2$ if and only if $\kappa$ is one of the following $(n-1)$-tuples: 
\begin{gather*}
(0,\dots,  0),  \quad (1,0,  \dots,  0),  \quad (0,\dots, 0,1), \quad (0,1,0,\dots,  0),  \\ 
(0,\dots,  0,1,0),  \quad (2,0,\dots,  0),  \quad (0,\dots,0,2),  \quad (1,0,\dots, 0,1). 
\end{gather*}
If $\kappa$ is one of these $(n-1)$-tuples, then $\dim \W_\kappa \le n^2$ by a direct calculation using Proposition~\ref{prop:dim-formula}.  To show that $\dim \W_{\kappa} > n^2$,  we split into three cases:
\begin{enumerate}[(i)]
\item If $\kappa = (\kappa_1,0, \dots, 0,\kappa_{n-1})$ where $\kappa_1,\kappa_{n-1} \ge 1$ and $\kappa_1 + \kappa_{n-1} \ge 3$,  then by applying Lemma~\ref{lem:dimcut} and \eqref{eq:dimSym} recursively, we get
\[
\dim \W_\kappa \ge \dim \W_{2,0,\dots,0,1} = \frac{n(n-1)(n+2)}{2} > n^2.
\]
\item If $\kappa = (\kappa_1,0,\dots,\kappa_{n-1})$ where either $\kappa_1 \ge 3$ and $\kappa_{n-1} = 0 $,  or  $\kappa_1 = 0$ and $\kappa_{n-1} \ge 3$, then Lemma~\ref{lem:dimcut} and \eqref{eq:dimSym} yield
\[\dim \W_\kappa \ge \dim \W_{3,0,\dots,0} = \binom{n+2}{3} > n^2.\]
\item If $\kappa = (\kappa_1,\dots,\kappa_{n-1})$ where either $\kappa_i \ge 1$ for some $3 \le i \le n-3$, or $\kappa_2 + \kappa_{n-2} \ge 2$ and $\kappa_i = 0$ for all $3 \le i \le n-3$,  then Lemma~\ref{lem:dimcut} gives
\[
\dim \W_\kappa \ge 
\begin{cases}
\dim \W_{e_i} = \binom{n}{i}  \quad &\text{if $\kappa_i \ge 1$ for some $3 \le i \le n-3$},    \\
\dim \W_{2e_2} = \frac{n^2(n+1)(n-1)}{12}  \quad &\text{if $\kappa_3 = \dots = \kappa_{n-3} = 0$ and $\kappa_2 \ge 2$}, \\
\dim \W_{2e_{n-2}} = \frac{n^2(n+1)(n-1)}{12}  \quad &\text{if $\kappa_3 = \dots = \kappa_{n-3} = 0$ and $\kappa_{n-2} \ge 2$}, \\
\dim \W_{e_2 + e_{n-2}} = \frac{n^2(n+1)(n-3)}{4}
\quad &\text{if $\kappa_3 = \dots = \kappa_{n-3} = 0$ and $\kappa_2 = \kappa_{n-2} = 1$}.
\end{cases} 
\]
Since $n \ge 9$,  we have $\dim \W_\kappa > n^2$.  \qedhere
\end{enumerate}
\end{proof}
Combining Lemma~\ref{lem:slnc} with the general results of Section~\ref{sec:man},  we obtain the following, where the notation $\P(\G_1,\G_2)$ is as in \eqref{eq:semidirect}.
\begin{theorem}[Classification of faithful representations of $\SL_n(\C)$-spaces]\label{thm:slnc-classification}
Let $n \ge 9$. Let $\H \subsetneq \SL_n(\C)$ be a closed subgroup and $\W$ be a complex $\SL_n(\C)$-module of dimension at most $n^2$.  The homogeneous space $\SL_n(\C)/\H$ admits a faithful representation in $\W$ if and only if there are integers $0 \le b \le n$; $0 \le c \le 2$;  $0 \le d, e \le 1$ such that $\bigl(  \frac{c + d }{2} + e - 1  \bigr) n^2 + \bigl( b - \frac{c}{2}  + \frac{d}{2} \bigr) n - e \le 0$ and
\begin{align}
\W &\cong \C^{n \times b} \oplus \Alt^2(\C^n)^{\oplus c} \oplus \Sym^2(\C^n)^{\oplus d} \oplus \sl_n(\C)^{\oplus e},  \label{eq:slnc-sum:1}\\
\H &= \H_1  \cap \biggl( \bigcap_{i=1}^c \H_{2,i} \biggr) \cap \biggl( \bigcap_{j=1}^d \H_{3,j} \biggr) \cap \biggl( \bigcap_{k=1}^e \H_{4,k} \biggr) , \label{eq:slnc-sum:2}
\end{align}
where
\begin{align*}
\H_1 &= Q \P\bigl(I_r, \SL_{n-r}(\C)\bigr) Q^{\ast},  &\H_{2,i} &= U_i \P\bigl(\Sp_{2r_i}(\C;\Omega_i), \SL_{n-2r_i}(\C)\bigr) U_i^{\ast}, \\
\H_{3,j} &= V_j \S\Bigl( \P\bigl(\O_{s_j}(\C;B_j), \GL_{n-s_j}(\C)\bigr) \Bigr) V_j^{\ast}, &\H_{4,k} &= A_k \S\Bigl( \prod_{u=1}^p \UT_{n_{k,u,1},\dots,  n_{k,u,t_u}} \Bigr) A_k^{-1}, 
\end{align*}
for some integers $0 \le r \le b$;  $0 \le r_i \le n/2$;  $0 \le s_j \le n$;  $1 \le n_{k,u,1},\dots,  n_{k,u,t_u} \le n$ such that 
$\sum_{u=1}^p (n_{k,u,1} + \dots + n_{k,u,t_u}) = n$,  and matrices $Q,U_i , V_j \in \Un_n$,  $\Omega_i \in \Alt^2(\R^{2r_i}) \cap \GL_{2r_i}(\R)$,  $B_j \in \Sym^2(\R^{r_i}) \cap \GL_{r_j}(\R)$,  $A_k \in \GL_n(\C)$,  with $1 \le i \le c$,  $1 \le j \le d$, and $1 \le k \le e$.
\end{theorem}
\begin{proof}
If $\W$ and $\H$ satisfy \eqref{eq:slnc-sum:1} and \eqref{eq:slnc-sum:2} respectively,  then Proposition~\ref{prop:min-sum} and \eqref{eq:stabilizer} collectively imply that $\SL_n(\C)/\H$ has a faithful representation in $\W$.  For the converse, we apply Proposition~\ref{prop:min-sum} and Lemma~\ref{lem:slnc} to conclude that $\W$ must take the form in \eqref{eq:slnc-sum:1}.  The constraints on $b,c,d,e$ follows from the assumption that $\dim \W \le n^2$. Proposition~\ref{prop:min-sum} also shows that $\H$ is the intersection of stabilizers of  $\SL_n(\C)$ acting on  $\C^{n\times b}$ (via left multiplication),  $\Alt^2(\C^n)$ and  $\Sym^2(\C^n)$ (via congruence), and  $\sl_n(\C)$ (via the adjoint action, which reduces to similarity).  Thus \eqref{eq:slnc-sum:2} follows from Lemmas~\ref{lem:rectangle}--\ref{lem:JCF} and \eqref{eq:stabilizer}.
\end{proof}

\subsection{Faithful linear representations of $\SL_n(\R)$-spaces}\label{sec:slnr}

Every complex irreducible $\SL_n(\C)$-module can be written as $\U \otimes \C$ for some real irreducible $\SL_n(\R)$-module \cite[Proposition~26.23]{Fulton_Harris}. This yields the following corollary of Lemma~\ref{lem:slnc}.
\begin{corollary}[Low dimension $\SL_n(\R)$-modules]\label{cor:slnr}
Let $n\ge 9$.  A real irreducible $\SL_n(\R)$-module $\U_{\kappa}$ has dimension at most $n^2$ if and only if it is isomorphic to $\R$,  $\R^n$,  $\Alt^2(\R^n)$,  $\Sym^2(\R^n)$, or $\sl_n(\R)$.
\end{corollary}
Applying the same arguments in the proof of Theorem~\ref{thm:slnc-classification} verbatim,  we arrive at the required classification.
\begin{proposition}[Classification of faithful representations of $\SL_n(\R)$-spaces]\label{prop:slnr-classification}
Let $n \ge 9$. Let $\H \subsetneq \SL_n(\R)$ be a closed subgroup and $\U$ be a real $\SL_n(\R)$-module of dimension at most $n^2$.  The homogeneous space $\SL_n(\R)/\H$ admits a faithful representation in $\U$ if and only if there are integers $0 \le b \le n$; $0 \le c \le 2$;  $0 \le d,e \le 1$ such that $\bigl(  \frac{c + d }{2} + e - 1  \bigr) n^2 + \bigl( b - \frac{c}{2}  + \frac{d}{2} \bigr) n - e \le 0$ and
\begin{align*}
\U &\cong \R^{n \times b} \oplus \Alt^2(\R^n)^{\oplus c} \oplus \Sym^2(\R^n)^{\oplus d} \oplus \sl_n(\R)^{\oplus e},  \\
\H &= \H_1  \cap \biggl( \bigcap_{i=1}^c \H_{2,i} \biggr) \cap \biggl( \bigcap_{j=1}^d \H_{3,j} \biggr) \cap \biggl( \bigcap_{k=1}^e \H_{4,k} \biggr),
\end{align*}
where
\begin{align*}
\H_1 &= Q \P\bigl(I_r , \SL_{n-r}(\R)\bigr) Q^{\ast}, \qquad
\H_{2,i} = U_i \P\bigl(\Sp_{2r_i}(\R;\Omega_i), \SL_{n-2r_i}(\R)\bigr) U_i^{\ast},   \\
\H_{3,j} &= V_j \S\Bigl( \P\bigl(\O_{s_j}(\R;B_j), \GL_{n-s_j}(\R)\bigr) \Bigr) V_j^{\ast},  \\ 
\H_{4,k} &= A_k \S\Bigl( \prod_{u=1}^p \UT_{n_{k,u,1},\dots,  n_{k,u,t_u}} \times \prod_{v=1}^q \BUT_{m_{k,v,1},\dots, m_{k,v,l_v}} \Bigr) A_k^{-1},
\end{align*}
for some integers $0 \le r \le b$; $0 \le r_i \le n/2$; $0 \le s_j \le n$; $1 \le n_{k,u,1},\dots, n_{k,u,t_u},  m_{k,v,1},\dots, m_{k,v,l_v} \le n$ such that $\sum_{u=1}^p (n_{k,u,1} + \dots +  n_{k,u,t_u})  + \sum_{v=1}^q (m_{k,v,1} + \dots + m_{k,v,l_v}) = n$ and matrices $Q, U_i, V_j \in \O_n(\R)$,  $\Omega_i \in \Alt^2(\R^{2r_i}) \cap \GL_{2r_i}(\R)$,  $B_j \in \Sym^2(\R^{r_i}) \cap \GL_{r_j}(\R)$,  $A_k \in \GL_n(\R)$ with  $1 \le i \le c$,  $1 \le j \le d$, and $1 \le k \le e$.
\end{proposition}
\subsection{Faithful linear representations of $\SU_n$-spaces}\label{sec:sun}
Unlike the case of $\SL_n(\R)$,  only some irreducible complex $\SU_n$-modules can be obtained from the real ones.  We reproduce \cite[Proposition~26.24]{Fulton_Harris} for easy reference:
\begin{lemma}\label{lem:su-real} 
Let $\W_\kappa$ be an irreducible complex representation of $\SL_n(\C)$ with $\kappa = (\kappa_1,\dots,  \kappa_{n-1}) \in \N^{n-1}$.  There exists an irreducible real $\SU_n$-module $\U$ such that $\W_\kappa = \U \otimes \C$ if and only if $\kappa_i = \kappa_{n-i}$ for all $i=1,\dots, n-1$ and one of the following holds:
\begin{enumerate}[(a)]
    \item $n \equiv 1 \pmod 2$,
    \item $n \equiv 0 \pmod 4$,
    \item $n \equiv 2 \pmod 4$ and $\kappa_{n/2}  \equiv 0 \pmod 2$.
\end{enumerate}
\end{lemma}
In particular,  if $\W \le n^2$ and $\W = \U \otimes \C$ for some real $\SU_n$-module $\U$,  then it follows from Lemmas~\ref{lem:slnc} and \ref{lem:su-real} that $\U$ is isomorphic to either $\R$ or $\mathfrak{su}_n$.
\begin{proposition}[Classification of faithful representations of $\SU_n$-spaces]\label{prop:sun-classification}
Let $n \ge 9$. Let $\H \subsetneq \SU_n$ be a closed subgroup and $\U$ be a real $\SU_n$-module such that $\dim \U \le n^2$.  Then $\SU_n/\H$ admits a faithful representation in $\U$ if and only if 
\[
\U \cong \mathfrak{su}_n,\quad \H =Q \S (\Un_{k_1} \times \dots \times \Un_{k_d} ) Q^\ast,
\] 
where $Q\in \SU_n$,  $1 \le k_1,\dots,  k_d \le n$, and $k_1 + \dots + k_d = n$.
\end{proposition}
\begin{proof}
By Lemma~\ref{lem:su-real}, $\U$ is either a direct sum of trivial modules or the module of adjoint representation.  If $\U \cong \R^{\oplus b}$ for some $ b \in \{1,\dots, n\}$,  then $\SU_n/\H$ is a single point since $\SU_n$ acts trivially on $\U$. In this case $\H = \SU_n$.  If $\U \cong \mathfrak{su}_n$,  then by the spectral theorem,  we may write any $X \in \mathfrak{su}_n$ as $X =i Q \diag (\lambda_1 I_{k_1},  \dots,  \lambda_d I_{k_d}) Q^\ast$,  where $Q\in \SU_n$,  $k_1 + \dots + k_d = n$, and $\lambda_1,  \dots,  \lambda_d \in \R$ are distinct with $\lambda_1 + \dots + \lambda_d = 0$.  A direct calculation shows that $(\SU_n)_X = Q \S (\Un_{k_1} \times \dots \times \Un_{k_d} ) Q^\ast$. This together with Lemma~\ref{lem:embedding} completes the proof.  
\end{proof} 

\section{Faithful linear representations of Type~$\mathrm{B}$ and $\mathrm{D}$ spaces}\label{sec:BD}

In this section we will find all homogeneous spaces $\G/\H$ that have faithful representations, where $\G = \SO_n(\C)$ or  its real form $\SO_{p,n-p}$,  $p =0,1, \dots, n$. 

\subsection{Faithful linear representations of $\SO_n(\C)$-spaces}\label{sec:sonc}

We have the following analogue of Lemma~\ref{lem:slnc} for $\SO_n(\C)$.
\begin{lemma}[Low dimension $\SO_n(\C)$-modules]\label{lem:sonc}
Let $n \ge 19$.  Let $\W$ be an irreducible complex $\SO_n(\C)$-module.  The dimension of $\W$ is at most $n^2$ if and only if $\W$ is isomorphic to $\C$,  $\C^n$,  $\Alt^2(\C^n) $, or $\Sym_\oh^2(\C^n) $. 
\end{lemma}
\begin{proof}
Since $\W$ is irreducible,  $\W \cong \W_{\kappa}$ for some $\kappa = (\kappa_1,\dots,  \kappa_m) \in \N^m$, $m \coloneqq \lfloor n/2 \rfloor$.  It suffices to show that $\dim \W_{\kappa} \le n^2$ if and only if  
\begin{enumerate}[(a)]
\item $n = 2m$ and $\kappa \in \N^m$ is one of the following $m$-tuples: 
\[
(0,\dots,  0),  \; (1,0,  \dots,  0),  \; (0,\dots, 0,1,0,0),  \; (0,1,0,\dots,  0),  \; (0,\dots,  0,1,0,0,0),  \;(2,0,\dots,  0);
\]
\item $n = 2m+1$ and $\kappa \in \N^m$ is one of the following $m$-tuples: 
\[
(0,\dots,  0),  \; (1,0,  \dots,  0),  \; (0,\dots, 0,1,0),  \; (0,1,0,\dots,  0),  \; (0,\dots,  0,1,0,0),  \;(2,0,\dots,  0).   
\]
\end{enumerate} 
By Proposition~\ref{prop:dim-formula},  we verify that $\dim \W_\kappa \le n^2$ when $\kappa$ is one of the $m$-tuples above. It remains to show the converse. Let $\dim \W_\kappa \le n^2$.  We first show that $\kappa_m  = 0$.  If $\kappa_m  \ge 1$,  then Lemma~\ref{lem:dimcut} together with the assumption that $n \ge 19$ leads to  a contradiction: 
\[ 
\dim \W_\kappa \ge \dim \W_{0,\dots,0,1} = 2^m > n^2.
\]
Next we observe that $\kappa_{m-1} = 0$ when $n = 2m$; otherwise Lemma~\ref{lem:dimcut} gives
\[
\dim \W_\kappa \ge \dim \W_{0,\dots,0,1,0}  = 2^{m-1} > n^2.
\] 
Thus we obtain 
\[
\kappa = \begin{cases}
    (\kappa_1,\dots,\kappa_{m-2},0,0) &\text{if } n = 2m,  \\
    (\kappa_1,\dots,\kappa_{m-1},0) &\text{if } n = 2m+1.
\end{cases}
\]
The rest of the proof is similar to that of Lemma~\ref{lem:slnc}, using Proposition~\ref{prop:dim-formula}.
\end{proof}

\begin{theorem}[Classification of faithful representations of $\SO_n(\C)$-spaces]\label{thm:sonc-classification}
Let $n \ge 19$. Let $\H \subsetneq \SO_n(\C)$ be a closed subgroup and $\W$ be a complex $\SO_n(\C)$-module of dimension at most $n^2$.  The homogeneous space $\SO_n(\C)/\H$ admits a faithful representation in $\W$ if and only if there exist integers $0\le b \le n$, $0\le c,d \le 2$ satisfying $(\frac{c+d}{2} - 1) n^2 + (b + \frac{d-c}{2})n - d \le 0$ and 
\[
\W \cong \C^{n \times b}\oplus \Alt^2(\C^n)^{\oplus c} \oplus \Sym_\oh^2(\C^n)^{\oplus d},  \quad
\H =\SO_n(\C) \cap  \H_1 \cap \biggl(\bigcap_{i=1}^c \H_{2,i} \biggr) \cap \biggl(\bigcap_{j=1}^d \H_{3,j} \biggr), 
\]
where 
\begin{align*}
\H_1 &= Q \P\bigl(I_r, \SL_{n-r}(\C)\bigr)Q^\ast, \quad
\H_{2,i} = U_i \P\bigl(\Sp_{2r_i}(\C; \Omega_i),  \SL_{n-2r_i}(\C)\bigr) U_i^\ast,   \\
\H_{3,j} &= V_j \P\bigl(\O_{s_j}(\C; B_j),  \GL_{n-s_j}(\C)\bigr) V_j^\ast, 
\end{align*}  
for some integers $0 \le r \le b$,  $0 \le r_i \le n/2$,  $0 \le s_j \le n$ and matrices $Q, U_i, V_j \in \Un_n$,  $\Omega_i \in \Alt^2(\R^{2r_i}) \cap \GL_{2r_i}(\R)$,  $B_j \in \S^2(\R^{s_j}) \cap \GL_{s_j}(\R)$ with $1 \le i \le c$ and $1 \le j \le d$.
\end{theorem}
\begin{proof}
By Proposition~\ref{prop:min-sum} and Lemma~\ref{lem:sonc},  $\W$ and $\H$ take the forms above.  The constraints on the integers $b$, $c$, and $d$ follow from $\dim \W \le n^2$.  To determine $\H_1,  \H_{2,i}$ and $\H_{3,j}$,  we compute the stabilizers of matrices in $\C^{n\times b}$,  $\Alt^2(\C^n)$, and $\Sym_\oh^2(\C^n)$ under left multiplication and congruence by $\GL_n(\C)$, using Lemmas~\ref{lem:rectangle}--\ref{lem:symmetric}.  
\end{proof}

\subsection{Faithful linear representations of $\SO_{p,n-p}$-spaces}\label{sec:sopq}

A complex module $\W$ of a real Lie group $\G$ is \emph{real} if $\W = \U \otimes \C$ for some real module $\U$,  and  is  \emph{quaternionic} if it is contained in a quaternionic module \cite[Section~26.3]{Fulton_Harris}.

\begin{lemma}[Low dimension $\SO_{p,n-p}$-module]\label{lem:sopq}
Let $n \ge 19$ and $0 \le p \le n$.  Let $\U$ be an irreducible real $\SO_{p,n-p}$-module.  The dimension of $\U$ is at most $n^2$ if and only if $\U$ isomorphic to $\R$,  $\R^n$,  $\Alt^2(\R^n;I_{p,n-p})$, or $\Sym^2_\oh(\R^n;I_{p,n-p})$.
\end{lemma}
\begin{proof}
Let $\U_0 \in \{ \R , \R^n,  \Alt^2(\R^n,I_{p,n-p}), \Sym^2_\oh(\R^n,I_{p,n-p}) \}$.  We first verify that 
\[
\W_0 \coloneqq \U_0 \otimes \C \in \{ \C,  \C^n,  \Alt^2(\C^n), \Sym^2_\oh(\C^n) \},
\]
showing that $\C$,  $\C^n$,  $\Alt^2(\C^n)$, $\Sym^2_\oh(\C^n)$ are real in the sense of the comment preceding this lemma. This is clearly the case if $\U_0 = \R$ or $\R^n$. For $\U_0 = \Alt^2(\R^n,I_{p,n-p})$, consider the complex linear map
\[
\varphi: \W_0 = \Alt^2(\R^n,I_{p,n-p}) \oplus i \Alt^2(\R^n,I_{p,n-p}) \to \Alt^2(\C^n),\quad \varphi(X_1,  i X_2) = I_{p,n-p}(X_1 + i X_2).
\]
Since $(X_1 + i X_2)^\tp I_{p,n-p}= -I_{p,n-p}(X_1 + i X_2)$,  $\varphi$ is well-defined.  A direct calculation verifies that $\varphi$ is bijective and $\SO_n(\C)$-equivariant.  Thus $\W_0 \cong \Alt^2(\C^n)$.  The proof for $\U_0 = \Sym^2_\oh(\R^n,I_{p,n-p})$ is nearly identical and thus omitted.  

Next,  let $\U$ be an irreducible real $\SO_{p,n-p}$-module of dimension at most $n^2$.  By Lemma~\ref{lem:sonc},  $\W \coloneqq \U \otimes \C$ is isomorphic to a direct sum of $\C$,  $\C^n$,  $\Alt^2(\C^n)$, and $\Sym^2_\oh(\C^n)$.  By Schur's Lemma,  the algebra $\End_{\SO_{p,n-p}} (\U)$ consisting of $\SO_{p,n-p}$-equivariant endomorphisms of $\U$ is a division algebra over $\R$.  Therefore  $\End_{\SO_{p,n-p}} (\U) \cong \R$, $\C$, or $\mathbb{H}$, and, accordingly,  $\W$ is either irreducible, splits into the direct sum of two conjugate modules, or splits into the direct sum of two isomorphic modules. Since $\C$,  $\C^n$,  $\Alt^2(\C^n)$, and $\Sym^2_\oh(\C^n)$ are all self-conjugate modules, we rule out $\End_{\SO_{p,n-p}} (\U) \cong \C$. So $\End_{\SO_{p,n-p}} (\U) \cong \R$ or $\mathbb{H}$ and  
\[
\W = \C^{\oplus a} \oplus (\C^n)^{\oplus b} \oplus \Alt^2(\C^n)^{\oplus c} \oplus \Sym^2_\oh(\C^n)^{\oplus d},
\]
where either $a + b + c + d = 1$ or one of $a,b,c,d$ is $2$ and the others are $0$. Suppose the latter holds. Then $\W = \W_1^{\oplus 2}$ for some $\W_1 \in \{ \C,  \C^n,  \Alt^2(\C^n), \Sym^2_\oh(\C^n) \} $, implying that $\End_{\SO_{p,n-p}} (\U) \cong \mathbb{H}$. But this in turn implies that $\W_1$ is quaternionic, contradicting the fact that the modules $\C$,  $\C^n$,  $\Alt^2(\C^n)$, $\Sym^2_\oh(\C^n)$ are all real. We conclude that $\W$ is one of $\C$,  $\C^n$,  $\Alt^2(\C^n)$,  or $\Sym^2_\oh(\C^n) $ and thus $\U$ is one of $\R$,  $\R^n$,  $\Alt^2(\R^n;I_{p,n-p})$, or $\Sym^2_\oh(\R^n;I_{p,n-p})$.
\end{proof}
In conjunction with Lemma~\ref{lem:sopq}, we repeat our argument in the proof of Theorem~\ref{thm:sonc-classification}, noting that the unitary matrices $P$,  $Q_i$,  $R_j$ therein are now orthogonal, which brings us to the following classification.
\begin{proposition}[Classification of faithful representations of $\SO_{p,n-p}$-spaces]\label{prop:sopq-classification}
Let $n \ge 19$. Let $\H \subsetneq \SO_{p,n-p}$ be a closed subgroup and $\U$ be a real $\SO_{p,n-p}$-module of dimension at most $n^2$.  The homogeneous space $\SO_{p,n-p}/\H$ admits a faithful representation in $\U$ if and only if there exist integers $0\le b \le n$; $ 0\le c,d \le 2$ satisfying $(\frac{c+d}{2} - 1) n^2 + (b + \frac{d-c}{2})n - d \le 0$ and 
\[
\U \cong \R^{n \times b}\oplus \Alt^2(\R^n;I_{p,n-p})^{\oplus c} \oplus \Sym_\oh^2(\R^n;I_{p,n-p})^{\oplus d},  \quad
\H =\SO_{p,n-p}  \cap  \H_1 \cap \biggl(\bigcap_{i=1}^c \H_{2,i} \biggr) \cap \biggl(\bigcap_{j=1}^d \H_{3,j}\biggr),  
\]
where
\begin{align*}
\H_1 &= Q \P\bigl(I_r, \SL_{n-r}(\R)\bigr) Q^\tp, 
\quad \H_{2,i} = U_i \P\bigl(\Sp_{2r_i}(\R; \Omega_i),  \SL_{n-2r_i}(\R)\bigr) U_i^\tp ,    \\
\H_{3,j} &= V_j \P\bigl(\O_{s_j}(\R; B_j),  \GL_{n-s_j}(\R)\bigr) V_j^\tp,
\end{align*}  
for some integers $0 \le r \le b$,  $0 \le r_i \le n/2$,  $0 \le s_j \le n$ and matrices $Q,U_i,V_j \in \O_n(\R)$,  $\Omega_i  \in \Alt^2(\R^{2r_i}) \cap \GL_{2r_i}(\R)$,  $B_j \in \S^2(\R^{s_j}) \cap \GL_{s_j}(\R)$,  with $1 \le i \le c$ and $1 \le j \le d$.
\end{proposition}
Proposition~\ref{prop:sopq-classification} includes the important special case $\SO_n(\R) = \SO_{n,0}$; but in this case there is some simplification due to the following:  For any $Q \in \O_n(\R)$ and any subgroup $\G \subseteq \GL_n(\R)$,  it holds that $Q \G Q^\tp \cap \SO_n(\R) =Q \bigl( \G \cap \SO_{n}(\R) \bigr)Q^\tp$.
\begin{corollary}[Classification of faithful representations of $\SO_n(\R)$-spaces]\label{cor:sonr-classification}
Let $n \ge 19$. Let $\H \subsetneq \SO_n(\R)$ be a closed subgroup and $\U$ be a real $\SO_n(\R)$-module of dimension at most $n^2$.  The homogeneous space $\SO_n(\R)/\H$  admits a faithful representation in $\U$ if and only if there exist nonnegative integers $b \le n$; $c,d \le 2$ satisfying $(\frac{c+d}{2} - 1) n^2 + (b + \frac{d-c}{2})n - d \le 0$, and 
\[
\U \cong \R^{n \times b}\oplus \Alt^2(\R^n)^{\oplus c} \oplus \Sym_\oh^2(\R^n)^{\oplus d},  \quad
\H = \H_1 \cap \biggl(\bigcap_{i=1}^c \H_{2,i}\biggr) \cap \biggl(\bigcap_{j=1}^d \H_{3,j}\biggr),
\]
where
\begin{align*}   
\H_1 &= Q \bigl( \{I_r\} \times \SO_{n-r}(\R) \bigr) Q^\tp, \quad
\H_{2,i} = U_i \bigl( \Sp_{2r_i}(\R; \Omega_i) \cap \SO_{2r_i}(\R)  \times  \SO_{n-2r_i}(\R) \bigr) U_i^\tp,   \\
\H_{3,j} &= V_j \bigl(  \S(\O_{s_j}(\R; B_j) \times \O_{n-s_j}(\R)) \bigr) V_j^\tp,
\end{align*}  
for some integers $0 \le r \le b$,  $0 \le r_i \le n/2$,  $0 \le s_j \le n$ and matrices $Q, U_i, V_j \in \O_n(\R)$,  $\Omega_i \in \Alt^2(\R^{2r_i}) \cap \GL_{2r_i}(\R)$, $B_j \in \S^2(\R^{s_j}) \cap \GL_{s_j}(\R)$, with  $1 \le i \le c$ and $1 \le j \le d$.
\end{corollary}

\section{Faithful linear representations of Type~$\mathrm{C}$ spaces}\label{sec:C}

In this section we will find all homogeneous spaces $\G/\H$ that have faithful representations, where $\G$ is the complex,  real, or compact symplectic group.  We begin by observing that for any $Q \in \GL_{2n}(\F)$ and $\G \subseteq \GL_{2n}(\F)$,
\[
Q\G Q^{-1} \cap \Sp_{2n}(\F;\Omega) = \G \cap \Sp_{2n}(\F; Q^\tp \Omega Q).
\]
As a reminder of the notations set in Section~\ref{sec:gp}, we write $J_{2n} \coloneqq \begin{bsmallmatrix}
0 & I_n \\
-I_n & 0
\end{bsmallmatrix}$ and
\begin{align*}
\Sym^2_\oh(\C^{2n};J_{2n}) &= \biggl\{\begin{bmatrix}
  X &D\\
  C &X^\tp
\end{bmatrix} \in \C^{2n \times 2n} : C,D \in \Alt^2(\C^n), \, X \in \sl_n(\C)\biggr\},  \\
\Alt^2(\C^{2n}; J_{2n}) &= \biggl\{\begin{bmatrix}
  X &D\\
  C & -X^\tp
\end{bmatrix} \in \C^{2n \times 2n} : C,D \in \Sym^2(\C^n)\biggr\}.
\end{align*}

\subsection{Faithful linear representations of $\Sp_{2n}(\C)$-spaces}\label{sec:spnc}

By inductively applying Lemma~\ref{lem:dimcut} and Proposition~\ref{prop:dim-formula},  we obtain the following characterization of irreducible $\Sp_{2n}(\C)$-modules.
\begin{lemma}[Low dimension $\Sp_{2n}(\C)$-modules]\label{lem:spnc}
Let $n \ge 5$.  An irreducible complex $\Sp_{2n}(\C)$-module $\W$ has dimension at most $4n^2$ if and only if it is isomorphic to $\C$,  $\C^{2n}$,  $\Sym^2_\oh(\C^{2n};J_{2n})$, or $\Alt^2(\C^{2n}; J_{2n})$. 
\end{lemma}
The same argument used in the proofs of Theorems~\ref{thm:slnc-classification} and \ref{thm:sonc-classification}, in conjunction with Lemma~\ref{lem:spnc},   gives us the following classification.
\begin{theorem}[Classification of faithful representations of $\Sp_n(\C)$-spaces]\label{thm:spnc-classification}
Let $n \ge 5$. Let $\H \subsetneq \Sp_{2n}(\C)$ be a closed subgroup and $\W$ be a complex $\Sp_{2n}(\C)$-module of dimension at most $4n^2$.  Then the homogeneous space $\Sp_{2n}(\C)/\H$ admits a faithful representation in $\W$ if and only if there are nonnegative integers $b \le 2n$; $c,d \le 1$ satisfying $(c+d-2)n^2 + (b - \frac{c-d}{2})n - \frac{c}{2} \le 0$ and
\[
\W = \C^{2n \times b}\oplus \Sym^2_\oh(\C^{2n},J_{2n})^{\oplus c} \oplus \Alt^2(\C^{2n},J_{2n})^{\oplus d},\quad \H = \Sp_{2n}(\C) \cap \H_1 \cap \biggl( \bigcap_{i=1}^c \H_{2,i} \biggr) \cap  \biggl( \bigcap_{j=1}^d \H_{3,d} \biggr),
\]
where 
\begin{align*}
\H_1 &= Q \P\bigl(I_r,  \GL_{2n-r}(\C)\bigr) Q^\ast,\quad \H_{2,i} = U_i \P\bigl(\Sp_{2r_i}(\C;\Omega_i),  \GL_{2(n-r_i)}(\C)\bigr)U_i^\ast \\
\H_{3,j} &= V_j \P\bigl(\O_{s_j}(\C,B_j), \GL_{2n - s_j}(\C)\bigr) V_j^\ast ,
\end{align*} 
for some nonnegative integers $r \le b$,  $r_i \le n$,  $s_j \le 2n$ and matrices $Q, U_i, V_j \in \Un_n$,  $\Omega_i \in \Alt^2(\R^{2r_i}) \cap \GL_{2r_i}(\R)$,  $B_j \in \S^2(\R^{s_j}) \cap \GL_{s_j}(\R)$, with  $1 \le i \le c$ and $1 \le j \le d$.
\end{theorem} 

\subsection{Faithful linear representations of $\Sp_{2n}(\R)$-spaces}\label{sec:spnr}

Every complex $\Sp_n(\R)$-module $\W$ can be written as $\W = \U \otimes \C$ for some real $\Sp_{2n}(\R)$-module \cite[Proposition~26.23]{Fulton_Harris}.  This yields the following classification.

\begin{proposition}[Classification of faithful representations of $\Sp_{2n}(\R)$-spaces]\label{prop:spnr-classification}
Let $n \ge 5$. Let $\H \subsetneq \Sp_{2n}(\R)$ be a closed subgroup and $\U$ be a real $\Sp_{2n}(\R)$-module of dimension at most $4n^2$. Then the homogeneous space $\Sp_{2n}(\R)/\H$ admits a faithful representation in $\U$ if and only if there are nonnegative integers $b \le 2n$; $c,d \le 1$ satisfying $(c+d-2)n^2 + (b - \frac{c-d}{2})n - \frac{c}{2} \le 0$ and
\[
\U \cong \R^{2n \times b}\oplus \Sym^2_\oh(\R^{2n},J_{2n})^{\oplus c} \oplus \Alt^2(\R^{2n},J_{2n})^{\oplus d},\quad \H = \Sp_{2n}(\R) \cap \H_1 \cap \biggl( \bigcap_{i=1}^c \H_{2,i} \biggr) \cap  \biggl( \bigcap_{j=1}^d \H_{3,d} \biggr),
\]
where 
\begin{align*}
\H_1 &= Q \P\bigl(I_r,  \GL_{2n-r}(\R)\bigr) Q^\ast,\quad \H_{2,i} = U_i \P\bigl(\Sp_{2r_i}(\R;\Omega_i),  \GL_{2(n-r_i)}(\R)\bigr)U_i^\ast \\
\H_{3,j} &= V_j \P\bigl(\O_{s_j}(\R,B_j), \GL_{2n - s_j}(\R)\bigr) V_j^\ast ,
\end{align*} 
for some nonnegative integers $r \le b$,  $r_i \le n$,  $s_j \le 2n$ and matrices $Q, U_i, V_j \in \O_n(\R)$,  $\Omega_i \in \Alt^2(\R^{2r_i}) \cap \GL_{2r_i}(\R)$,  $B_j \in \S^2(\R^{s_j}) \cap \GL_{s_j}(\R)$, with  $1 \le i \le c$ and $1 \le j \le d$.
\end{proposition}

\subsection{Faithful linear representations of $\Sp_{2n}$-spaces}\label{sec:spn}

We will use the following fact for modules of $\Sp_{2n} = \Sp_{2n}(\C) \cap \SU_{2n}$, reproduced from \cite[Proposition~26.25]{Fulton_Harris} for easy reference.
\begin{lemma}\label{lem:spn}
    Let $\kappa = (\kappa_1,\dots,\kappa_n) \in \N^{n}$.  The complex irreducible $\Sp_n$-module $\W_\kappa$ can be written as $\W_{\kappa} = \U \otimes \C$ for some real module if and only if $\kappa_i$ is even for all odd $i \in \{1, \dots, n \}$.
\end{lemma}
So for  $\W \cong \C$ or $\C^{2n}$, there is no real $\Sp_{2n}$-module $\U$ such that $\W = \U \otimes \C$.  However,  we have  
\[
\Sym^2_\oh(\C^{2n};J_{2n}) = \bigl( \Sym_\oh^2(\C^{2n}; J_{2n}) \cap \su_{2n} \bigr) \otimes \C,  \quad
\Alt^2(\C^{2n}; J_{2n}) = \mathfrak{sp}_{2n} \otimes \C,
\]
where
\begin{align*}
\Sym_\oh^2(\C^{2n}; J_{2n}) \cap \su_{2n} &= \biggl\{\begin{bmatrix}
    C &D\\
    -D^* &C^\tp
\end{bmatrix} \in \C^{2n \times 2n} : D \in \Alt^2(\C^{n}),\, C \in \su_n\biggr\}, \\
\mathfrak{sp}_{2n} = \Alt^2(\C^{2n}; J_{2n}) \cap \mathfrak{su}_{2n} &=  \biggl\{\begin{bmatrix}
    C &D\\
    -D^* & -C^\tp
\end{bmatrix} \in \C^{2n \times 2n}: D \in \Sym^2(\C^{n}),\, C \in \mathfrak{u}_n \biggr\},
\end{align*}
which are real vector spaces of dimensions $2n^2-n-1$ and $2n^2+n$  respectively.  This leads us to the following classification.
\begin{proposition}[Classification of faithful representations of $\Sp_{2n}$-spaces]\label{prop:spn-classification}
Let $n \ge 5$. Let $\H \subsetneq \Sp_{2n}$ be a closed subgroup and $\U$ be a real $\Sp_{2n}$-module.  Then $\Sp_{2n} / \H$ admits a faithful representation in $\U$ if and only if there is a pair $(c,d) \in  \{(0,1), (1,0), (1,1)\}$ such that 
\[
\U \cong \bigl(\Sym^2_\oh(\C^n;J_{2n}) \cap \su_{2n}\bigr)^{\oplus c} \oplus \sp_{2n}^{\oplus d},  \qquad  \H = \Sp_{2n} \cap \H_c \cap \H_d,
\]
where 
\begin{align*}
\H_c &= \begin{cases} 
Q \P ( \Sp_{2r}(\C;\Omega), \GL_{2(n-r)}(\C) ) Q^\ast \quad &\text{if~} c = 1, \\
\GL_{2n}(\C) \quad &\text{if~} c = 0,
\end{cases} \\
\H_d &= 
\begin{cases}
U \P\bigl(\O_s(\C;B), \GL_{2n-s}(\C)\bigr) U^\ast \quad &\text{if~} d = 1, \\
\GL_{2n}(\C) \quad &\text{if~} d = 0,
\end{cases}
\end{align*}
for some nonnegative integers $r \le n$,  $s \le 2n$ and matrices $Q, U\in \Un_{2n}$,  $\Omega \in \Alt^2(\R^{2r}) \cap \GL_{2r}(\R)$,  $B \in \S^2(\R^{s}) \cap \GL_{s}(\R)$.
\end{proposition}

\section{Minimal faithful linear representations}\label{sec:min}

In this section, we will determine explicit faithful representations of various common, and some uncommon, $\G$-manifolds. Indeed, a few of these esoteric manifolds, like the indefinite Grassmannian or the last five flag manifolds in Table~\ref{tab:set}, have not appeared before in the literature. We are thus obliged to give a proper definition here, with the common ones included alongside for comparison and completeness. With this in mind, we define them as abstract manifolds in Table~\ref{tab:set}. Their alternative descriptions as homogeneous spaces have already appeared in Table~\ref{tab:homo}; and we will represent them as submanifolds of $\F^{n \times n}$ in Table~\ref{tab:mat}.

\begin{table}[htb]
\footnotesize
\tabulinesep=0.2ex
\begin{tabu}{@{}c|c|c|c|c}
  & manifold & notation & elements & form\\
\hline 
\multirow{12}{*}{\rotatebox[origin=c]{90}{Grassmann}} 
& real  &$\Gr(k,\R^n)$ &$k$-planes & \\
& complex &$\Gr(k,\C^n)$ &$k$-planes & \\
& quaternionic &$\Gr(k,\mathbb{H}^n)$ & right $k$-planes & \\ 
& real symplectic &$\Gr_{\Sp}(2k,\R^{2n})$ & symplectic $2k$-planes &symplectic\\
& complex symplectic &$\Gr_{\Sp}(2k,\C^{2n})$ & symplectic $2k$-planes &symplectic\\
& complex locus &$\Gr_\C(k, \R^n)$ & nondegenerate $k$-planes &symmetric \\
& special Lagrangian &$\operatorname{SLGr}(\R^{2n})$ &special Lagrangian planes &symplectic\\
& complex Lagrangian &$\operatorname{LGr}(\C^{2n})$ &Lagrangian planes &symplectic \\
& skew-Hermitian Lagrangian  &$\operatorname{SLGr}^*(\mathbb{H}^{2n})$ &special Lagrangian planes &skew-Hermitian \\
& orthogonal Lagrangian &$\operatorname{SOGr}(\mathbb{C}^{2n}) $  &special Lagrangian planes & symmetric \\
& isotropic &$\operatorname{IGr}(k,\C^n)$  & isotropic $k$-planes &symmetric \\
& indefinite &$\Gr(p,q,\R^m \oplus \R^n)$ &$(p,q)$-planes &$(m,n)$-indefinite\\
\hline 
\multirow{8}{*}{\rotatebox[origin=c]{90}{flag}}
& real &$\Fl(k_1,\dots,k_m,\R^n)$ &$(k_1,\dots,k_m)$-flags & \\
& complex &$\Fl(k_1,\dots,k_m,\C^n)$ &$(k_1,\dots,k_m)$-flags & \\
& quaternionic &$\Fl(k_1,\dots,k_m,\mathbb{H}^n)$ &$(k_1,\dots,k_m)$-flags & \\
& isotropic &$\operatorname{IFl}(k_1,\dots,k_m,\R^{2n})$ & symplectic $(2k_1,\dots,2k_m)$-flags & symplectic\\
& partial isotropic&$\operatorname{IFl}(k_1,\dots,k_m,\R^{2n+p})$ & symplectic $(2k_1,\dots,2k_m)$-flags & symplectic\\
& real symplectic &$\Fl_{\Sp}(2k_1,\dots,2k_m,\R^{2n})$ & symplectic $(2k_1,\dots,2k_m)$-flags & symplectic \\
& complex symplectic &$\Fl_{\Sp}(2k_1,\dots,2k_m,\C^{2n})$  & symplectic $(2k_1,\dots,2k_m)$-flags & symplectic\\
& complex Lagrangian &$\operatorname{LFl}(k_1,\dots,k_m,\C^{2n})$ &  isotropic $(k_1,\dots,k_m)$-flags & isotropic\\
\hline 
\multirow{5}{*}{\rotatebox[origin=c]{90}{Stiefel}}
& noncompact real & $\St_k(\R^n)$ & $k$-frames &\\
& noncompact complex & $\St_k(\C^n)$ & $k$-frames &\\
& compact real & $\Stiefel_k(\R^n)$ &orthonormal $k$-frames & symmetric definite\\
& compact complex & $\Stiefel_k(\C^n)$ &orthonormal $k$-frames & Hermitian definite \\
& compact quaternionic & $\Stiefel_k(\mathbb{H}^n)$ &orthonormal $k$-frames & Hermitian definite
\end{tabu}
\caption{Grassmann, flag, Stiefel manifolds defined set-theoretically. Some definitions require a nondegenerate bilinear form on the ambient space.} 
\label{tab:set}
\end{table}
Aside from the few aforementioned exceptions, these $\G$-manifolds have all appeared in the literature in various contexts, although some are rather obscure. We now proceed to find their minimal faithful representations, as defined in Definition~\ref{def:faith}.

We define the \emph{Mostow--Palais dimension} of a $\G$-manifold $\M$ as:
\[
d_\MP(\M) \coloneqq \min \{ \dim \V : \text{there exists a Mostow--Palais embedding }\varepsilon : \M \to \V \}.
\]
Note that there is no requirement for the $\G$-module $\V$ be a subspace of $\F^{n \times n}$ or for $\G$ to act via matrix multiplication like in \eqref{eq:act}. In principle, $d_\MP(\M)$ could be smaller than the dimension of the minimal faithful representation of $\M$. But we will show that they are equal for all $\M$ that is an irreducibly faithfully representable homogeneous space. 
\begin{theorem}[Characterization of minimal faithful representations]\label{thm:min}
Let $k \le n$ be positive integers and $\G$ be as in \eqref{eq:G}.
\begin{enumerate}[\upshape (a)]
\item\label{it:f1} A homogeneous space $\G/\H$ has a faithful representation if and only if $\H$ is the form in Theorem~\ref{thm:slnc-classification}, Proposition~\ref{prop:slnr-classification}, Proposition~\ref{prop:sun-classification}, Theorem~\ref{thm:sonc-classification}, Proposition~\ref{prop:sopq-classification}, Corollary~\ref{cor:sonr-classification}, Theorem~\ref{thm:spnc-classification}, Proposition~\ref{prop:spnr-classification}, or Proposition~\ref{prop:spn-classification}.

\item\label{it:f2} Let $\G/\H$ be faithfully representable. A representation $\varepsilon : \G/\H \to \V$ is a minimal faithful representation if and only if $\varepsilon = \varepsilon_1 \oplus \dots \oplus  \varepsilon_r: \G/\H \to \V \cong  \V_1 \oplus \dots \oplus \V_r$, where $\varepsilon_i:\G/\H \to \V_i$ is a representation in Table~\ref{tab:min}, $i=1,\dots, r$.
Furthermore,
\[
\dim(\V) = d_\MP(\G/\H).
\]
\end{enumerate}
\begin{table}[htb]
\footnotesize
\tabulinesep=0.2ex
\begin{tabu}{@{}c|c|c|c|c|c}
$\G$ & $n$ & $ \varepsilon_i(\G/\H)$ &  $\V_i$ & $\dim(\V_i)$ & $X$ \\
\hline 
$\SL_n(\F)$ & $n \ge 9$ & \makecell[c]{%
  $\{AX : A \in \SL_n(\F)\}$\\
  $\{AXA^\tp : A \in \SL_n(\F)\}$\\
  $\{AXA^\tp : A \in \SL_n(\F)\}$\\
  $\{AXA^{-1} : A \in \SL_n(\F)\}$%
} 
& \makecell[c]{%
  $\F^{n \times k}$\\
  $\Alt^2(\F^n)$\\
  $\Sym^2(\F^n)$\\
  $\sl_n(\F)$%
}
& \makecell[c]{%
  $nr$\\
  $\frac{1}{2}n(n-1)$\\
  $\frac{1}{2}n(n+1)$\\
  $n^2-1$%
}
& \makecell[c]{%
  full rank\\
  any\\
  any\\
  any%
}\\
\hline
$\SU_n$ & $n \ge 9$ & \makecell[c]{%
  $\{QX : Q \in \SU_n\}$\\
  $\{QXQ^\tp : Q \in \SU_n\}$\\
  $\{QXQ^\tp : Q \in \SU_n\}$\\
  $\{QXQ^\ast : Q \in \SU_n\}$\\
  $\{QXQ^\ast : Q \in \SU_n\}$%
} 
& \makecell[c]{%
  $\C^{n \times k}$\\
  $\Alt^2(\C^n)$\\
  $\Sym^2(\C^n)$\\
  $\sl_n(\C)$\\
  $\mathfrak{su}_n$%
}
& \makecell[c]{%
  $nr$\\
  $\frac{1}{2}n(n-1)$\\
  $\frac{1}{2}n(n+1)$\\
  $n^2-1$\\
  $n^2-1$%
}
& \makecell[c]{%
  full rank\\
  any\\
  any\\
  any\\
  any%
}\\
\hline
$\SO_n(\F)$ & $n \ge 19$ & \makecell[c]{%
  $\{QX : Q \in \SO_n(\F)\}$\\
  $\{QXQ^\tp : Q \in \SO_n(\F)\}$\\
  $\{QXQ^\tp : Q \in \SO_n(\F)\}$%
} 
& \makecell[c]{%
  $\F^{n \times k}$\\
  $\Alt^2(\F^n)$\\
  $\Sym^2_\oh(\F^n)$%
}
& \makecell[c]{%
  $nr$\\
  $\frac{1}{2}n(n-1)$\\
  $\frac{1}{2}(n+2)(n-1)$%
}
& \makecell[c]{%
  full rank\\
  any\\
  any%
}\\
\hline
$\SO_{p,n-p}$ & $n \ge 19$ & \makecell[c]{%
  $\{QX : Q \in \SO_{p,n-p}\}$\\
  $\{QXQ^{-1} : Q \in \SO_{p,n-p}\}$\\
  $\{QXQ^{-1} : Q \in \SO_{p,n-p}\}$%
} 
& \makecell[c]{%
  $\R^{n \times k}$\\
  $\Alt^2(\R^n;I_{p,n-p})$\\
  $\Sym^2_\oh(\R^n;I_{p,n-p})$%
}
& \makecell[c]{%
  $nr$\\
  $\frac{1}{2}n(n-1)$\\
  $\frac{1}{2}(n+2)(n-1)$%
}
& \makecell[c]{%
  full rank\\
  any\\
  any%
}\\
\hline
$\Sp_{2n}(\F)$ & $n \ge 5$ & \makecell[c]{%
  $\{QX : Q \in \Sp_{2n}(\F)\}$\\
  $\{QXQ^{-1} : Q \in \Sp_{2n}(\F)\}$\\
  $\{QXQ^{-1} : Q \in \Sp_{2n}(\F)\}$%
} 
& \makecell[c]{%
  $\F^{2n \times k}$\\
  $\Alt^2(\F^{2n}; J_{2n})$\\
  $\Sym^2_\oh(\F^{2n};J_{2n})$%
}
& \makecell[c]{%
  $2nr$\\
  $2n^2+n$\\
  $(n-1)(2n+1)$%
}
& \makecell[c]{%
  full rank\\
  any\\
  any%
}\\
\hline
$\Sp_{2n}$ & $n \ge 5$ & \makecell[c]{%
  $\{QX : Q \in \Sp_{2n}\}$\\
  $\{QXQ^\ast : Q \in \Sp_{2n}\}$\\
  $\{QXQ^\ast : Q \in \Sp_{2n}\}$%
} 
& \makecell[c]{%
  $\C^{2n \times k}$\\
  $\mathfrak{sp}_{2n}$\\
  $\Sym^2_\oh(\C^{2n};J_{2n}) \cap \mathfrak{su}_{2n}$%
}
& \makecell[c]{%
  $2nr$\\
  $2n^2+n$\\
  $(n-1)(2n+1)$%
}
& \makecell[c]{%
  full rank\\
  any\\
  any%
}
\end{tabu}
\caption{Irreducible faithful representations of $\G/\H$.} 
\label{tab:min}
\end{table}
\end{theorem} 
\begin{proof}
The results listed in \ref{it:f1} characterize all possible forms of $\H$ for $\G/\H$ to admit a faithful representation. In the proofs of these results, we computed all possible forms of faithful representations using Lemmas~\ref{lem:rectangle}--\ref{lem:JCF}. By  tracking the dimensions of $\G$-modules for each faithful representation, a representation is a minimal faithful one if and only if it takes the forms in Table~\ref{tab:min} with the required $X \in \V$.

Let $\varepsilon: \G/\H \to \V$ be a minimal faithful representation in Table~\ref{tab:min}. By Lemma~\ref{lem:slnc}, Corollary~\ref{cor:slnr}, Lemma~\ref{lem:su-real}, Lemma~\ref{lem:sonc}, Lemma~\ref{lem:sopq}, Lemma~\ref{lem:spnc}, Lemma~\ref{lem:spn}, the $\G$-module $\V$ has the lowest dimension among all $\G$-modules that admit a Mostow--Palais embedding of $\G/\H$, regardless of the a priori group actions. Consequently, $\dim(\V) = d_\MP(\G/\H)$.
\end{proof}

As a consequence of Theorem~\ref{thm:min}, we obtain minimal faithful representations of the $\G$-manifolds in Table~\ref{tab:set}, or, equivalently, the homogeneous spaces in Table~\ref{tab:homo}.

\begin{theorem}[Minimal faithful representations of common $\G$-manifolds]\label{thm:mat}
Let $k \le n$ and $k_0 \coloneqq 0 < k_1  <  \dots <  k_m < k_{m+1} \coloneqq n$. Let  $\lambda, \mu, \lambda', \mu', \lambda_1,\dots,\lambda_{m+1}, \mu_1,\dots,\mu_{m+1} \in \R$ satisfy
\begin{gather*}
\lambda \neq \mu, \quad \lambda' \neq \mu', \quad \lambda_i \neq \lambda_j \text{ for all }i\neq j,\\
k\lambda + (n-k)\mu=0, \quad (p+q)\lambda' + (m+n-p-q)\mu' = 0, \quad n_1\lambda_1 + \dots n_{m+1}\lambda_{m+1}=0.
\end{gather*}
Let $\Lambda \coloneqq \diag (\lambda_1 I_{n_1},  \dots,\lambda_{m+1} I_{n_{m+1}})$ and $\Xi \coloneqq \diag(\mu_1 J_{2n_1}, \dots, \mu_{m+1} J_{2n_{m+1}})$. Then the representations given in Table~\ref{tab:mat} are minimal faithful.
\begin{table}[htb]
\footnotesize
\tabulinesep=0.2ex
\begin{tabu}{@{}c|c|c}
$\M$ & $\varepsilon(\M)$ & $\V$ \\
\hline
$\Gr(k,\R^n)$ &$\{Q \diag(\lambda I_{k},  \mu I_{n-k}) Q^\tp : Q \in \SO_n(\R)\}$ &$\Sym^2_\oh(\R^n)$\\
$\Gr(k,\C^n)$ &$\{iQ \diag (\lambda I_{k},  \mu I_{n-k}) Q^* : Q \in \SU_n\}$ &$\su_n$\\
$\Gr(k,\mathbb{H}^n)$ &$\{iQ \diag (\lambda I_k, \mu I_{n-k}, \lambda I_k, \mu I_{n-k} ) Q^{\ast} : Q \in \Sp_{2n}\}$ &$\Sym^2_\oh(\C^{2n};J_{2n}) \cap \su_{2n}$ \\
$\Gr_{\Sp}(2k,\R^{2n})$ &$\{Q \diag(\lambda I_k, \mu I_{n-k}, \lambda I_k, \mu I_{n-k}) Q^{-1} : Q \in \Sp_{2n}(\R)\}$ &$\Sym^2_\oh(\R^{2n};J_{2n})$ \\
$\Gr_{\Sp}(2k,\C^{2n})$ &$\{Q \diag(\lambda I_k, \mu I_{n-k}, \lambda I_k, \mu I_{n-k}) Q^{-1} : Q \in \Sp_{2n}(\C)\}$ &$\Sym^2_\oh(\C^{2n}; J_{2n})$ \\
$\Gr_\C(k, \R^n)$ &$\{Q \diag(\lambda I_{k},  \mu I_{n-k}) Q^\tp : V \in \SO_n(\C)\}$ &$\Sym^2_\oh(\C^n)$\\
$\operatorname{SLGr}(\R^{2n})$ &$\{QQ^\tp : Q \in \mathsf{SU}_n\}$ &$\Sym^2(\C^n)$\\
$\operatorname{LGr}(\C^{2n})$ &$\{iQ\diag(I_n, -I_n)Q^* : Q \in \Sp_{2n}\}$ &$\Alt^2(\C^{2n}; J_{2n}) \cap \su_{2n}$\\
$\operatorname{SLGr^*(\mathbb{H}^{2n})}$ &$\{Q J_{2n} Q^\tp : Q \in \SU_{2n}\}$ &$\Alt^2(\C^{2n})$\\
$\operatorname{SOGr(\C^{2n})}$ &$\{Q J_{2n}Q^\tp : Q \in \SO_{2n}(\R)\}$ &$\Alt^2(\R^{2n})$\\
$\operatorname{IGr}(k,n)$ &$\{Q \diag(\lambda J_{2k},  0_{n - 2k }) Q^\tp : Q \in \SO_{n}(\R)\}$ &$\Alt^2(\R^{n})$\\
$\Gr(p,q,\R^{m+n})$ &$\{Q \diag(\lambda' I_p, \mu' I_{m-p},\lambda' I_q, \mu' I_{n-q})Q^{-1}: Q \in \SO_{m,n}\}$ &$\Sym^2_\oh(\R^{m+n};I_{m,n})$\\
\hline 
$\Fl(k_1,\dots,k_m,\R^n)$ &$\{Q \Lambda Q^\tp : Q \in \SO_n(\R)\}$ &$\Sym^2_\oh(\R^n)$ \\
$\Fl(k_1,\dots,k_m,\C^n)$ &$\{iQ \Lambda Q^* : Q \in \SU_n\}$ &$\su_n$ \\
$\Fl(k_1,\dots,k_m,\mathbb{H}^n)$ &$\{ iQ \diag(\Lambda, \Lambda) Q^* : Q \in \Sp_{2n}\}$ &$\Sym^2_\oh(\C^{2n}; J_{2n}) \cap \su_{2n}$\\
$\operatorname{IFl}(k_1,\dots,k_m,\R^{2n})$ &$\{Q \Xi Q^\tp : Q \in \SO_{2n}(\R)\}$ &$\Alt^2(\R^{2n})$ \\
$\operatorname{IFl}(k_1,\dots,k_m,\R^{2n+p})$ &$\{Q \diag(\Xi, 0_p) Q^\tp : Q \in \SO_{2n+p}(\R)\}$ &$\Alt^2(\R^{2n+p})$\\
$\Fl_{\Sp}(2k_1,\dots,2k_m,\mathbb{F}^{2n})$ &$\{Q \diag(\Lambda, \Lambda)Q^{-1} : Q \in \Sp_{2n}(\mathbb{F})\}$ &$\Sym^2_\oh(\mathbb{F}^{2n}; J_{2n})$\\
$\operatorname{LFl}(k_1,\dots,k_m,\mathbb{C}^{2n})$ &$\{ iQ \diag(\Lambda,-\Lambda) Q^\ast : Q \in \Sp_{2n} \}$ &$\mathfrak{sp}_{2n}$\\
\hline
$\St_k(\R^n)$ &$\bigl\{A\begin{bsmallmatrix} I_k\\ 0\end{bsmallmatrix}: A \in \SL_n(\R) \bigr\}$ &$\R^{n \times k}$ \\
$\St_k(\C^n)$ &$\bigl\{A\begin{bsmallmatrix} I_k\\ 0\end{bsmallmatrix}: A \in \SL_n(\C) \bigr\}$ &$\C^{n \times k}$ \\
$\Stiefel_k(\R^n)$ &$\bigl\{Q\begin{bsmallmatrix} I_k\\ 0\end{bsmallmatrix}: Q \in \SO_n(\R) \bigr\}$ &$\R^{n \times k}$\\
$\Stiefel_k(\C^n)$ &$\bigl\{Q\begin{bsmallmatrix} I_k\\ 0\end{bsmallmatrix}: Q \in \SU_n \bigr\}$ &$\C^{n \times k}$ \\
$\Stiefel_k(\mathbb{H}^n)$ &$\bigl\{Q\begin{bsmallmatrix} I_{2k}\\ 0\end{bsmallmatrix}: Q \in \Sp_n \bigr\}$ &$\C^{2n \times 2k}$ 
\end{tabu}
\caption{Minimal faithful representations for Tables~\ref{tab:homo} and \ref{tab:set}.} 
\label{tab:mat}
\end{table}
\end{theorem}
\begin{proof}
Using the homogeneous space characterizations of $\M$ in Table~\ref{tab:homo}, it follows from Theorem~\ref{thm:min}\ref{it:f1} that they have faithful representations.  The required minimality follows from Theorem~\ref{thm:min}\ref{it:f2}.  For concreteness, take $\Gr(k,\mathbb{H}^n)$ for illustration. From Table~\ref{tab:homo},  we have $\Gr(k,\mathbb{H}^n) \cong \Sp_{2n}/(\Sp_{2k} \times \Sp_{2n-2k})$.  If we take $(c,d) = (1,0)$,  $V = I_{2n}$,  and $\Omega = J_{2k}$ in Proposition~\ref{prop:spn-classification},  then
\[
\H = \Sp_{2n} \cap \H_1 = \Sp_{2n}(\C) \cap \Un_{2n} \cap \P\bigl(\Sp_{2k}(\C), \GL_{2n - 2k}(\C)\bigr) =  \Sp_{2k} \times \Sp_{2n - 2k}.
\]
From Table~\ref{tab:min}, there is an $\Sp_{2n}$-equivariant diffeomorphism $\Gr(k,\mathbb{H}^n) \cong \{ Q X Q^\ast: Q\in \Sp_{2n}\}$ for any $X \in \Sym^2_\oh(\C^{2n};J_{2n}) \cap \mathfrak{su}_{2n}$ whose stabilizer is isomorphic to $\H$ above.  Now we just check that we may pick $X = i\diag (\lambda I_k, \mu I_{n-k}, \lambda I_k, \mu I_{n-k} )$ where $\lambda, \mu \in \R$ are distinct and $k\lambda + (n-k)\mu = 0$. 
\end{proof} 

The values of $d_\MP(\M)$ now follow from Tables~\ref{tab:vec} and \ref{tab:mat}.
\begin{corollary}[Mostow--Palais dimensions]\label{cor:mp}
The lowest possible dimension of a Mostow--Palais embedding for any $\G$-manifold in Table~\ref{tab:set}, or, equivalently, its corresponding homogeneous space $\G/\H$ in Table~\ref{tab:homo}, is given in Table~\ref{tab:mp}.
\begin{table}[h]
\centering
\begin{minipage}[t]{0.45\textwidth}
\centering
\footnotesize
\tabulinesep=0.2ex
\begin{tabu}[t]{@{}c|c}
$\M$ & $d_\MP(\M)$\\
\hline
$\Gr(k,\R^n)$  &$\frac{1}{2}(n+2)(n-1)$\\
$\Gr(k,\C^n)$ &$n^2-1$\\
$\Gr(k,\mathbb{H}^n)$  &$(n-1)(2n+1)$\\
$\Gr_{\Sp}(2k,\R^{2n})$ &$(n-1)(2n+1)$\\
$\Gr_{\Sp}(2k,\C^{2n})$  &$(n-1)(2n+1)$\\
$\Gr_\C(k, \R^n)$ &$\frac{1}{2}(n+2)(n-1)$\\
$\operatorname{SLGr}(\R^{2n})$ &$\frac{1}{2}n(n+1)$\\
$\operatorname{LGr}(\C^{2n})$ &$2n^2+n$\\
$\operatorname{SLGr}^*(\mathbb{H}^{2n})$ &$n(2n-1)$\\
$\operatorname{SOGr}(\C^{2n})$ &$n(2n-1)$\\
$\operatorname{IGr}(k,n)$ &$\frac{1}{2}n(n-1)$\\
$\Gr(p,q,\R^{m+n})$ &$\frac{1}{2}(m+n+2)(m+n-1)$
\end{tabu}
\end{minipage}%
\begin{minipage}[t]{0.45\textwidth}
\centering
\footnotesize
\tabulinesep=0.2ex
\begin{tabu}[t]{@{}c|c}
$\M$ &$d_\MP(\M)$\\
\hline
$\Fl(k_1,\dots,k_m,\R^n)$ &$\frac{1}{2} (n+2)(n-1)$\\
$\Fl(k_1,\dots,k_m,\C^n)$ &$n^2-1$\\
$\operatorname{IFl}(k_1,\dots,k_m,\R^{2n})$ &$n(2n-1)$\\
$\operatorname{IFl}(k_1,\dots,k_m,\R^{2n+p})$ &$(n+\frac{p}{2})(2n+p-1)$\\
$\Fl_{\Sp}(2k_1,\dots,2k_m,\mathbb{F}^{2n})$ &$(n-1)(2n+1)$\\
$\operatorname{LFl}(k_1,\dots,k_m,\mathbb{C}^{2n})$ &$2n^2+n$\\
$\St_k(\R^n)$ &$nk$\\
$\St_k(\C^n)$ &$nk$\\
$\Stiefel_k(\R^n)$ &$nk$\\
$\Stiefel_k(\C^n)$ &$nk$\\
$\Stiefel_k(\mathbb{H}^n)$ &$4nk$
\end{tabu}
\end{minipage}
\caption{Mostow--Palais dimensions of the $\G$-manifolds in Table~\ref{tab:homo}.}
\label{tab:mp}
\end{table}
\end{corollary}

As the last column of Table~\ref{tab:homo} indicates, after accounting for the compact groups $\SU_n$ (type A), $\SO_n(\R)$ (type BD), $\Sp_n$ (type C), all ten types of classical irreducible compact symmetric spaces \cite{Helgason2001} appear in our classification. These play important roles for other classification results in areas far removed from differential geometry --- two notable examples are the classification of random matrix ensembles \cite{Altland_Zirnbauer_1997} and that of topological insulators and superconductors \cite{Schnyder_Ryu_Furusaki_Ludwig_2008}. It is even more remarkable that the \emph{Cartan embeddings} \cite[Proposition~3.42]{embedding} of these symmetric spaces turn out to be either identical or equivalent in the sense of Definition~\ref{def:equiv} to the minimal faithful representations that we found.
\begin{corollary}[Cartan embedding is a minimal faithful representation]\label{cor:car}
The Cartan embedding of any symmetric space in Table~\ref{tab:symm} is equivalent to its minimal faithful representation in Table~\ref{tab:mat}.
\begin{table}[h!]
\footnotesize
\tabulinesep=0.2ex
\begin{tabu}{@{}c|c|c}
type & $\G/\H$   & Cartan embedding $\varepsilon(\G/\H)$ \\
\hline
AI & $\SU_{n}/\SO_n(\R)$ & $\{ QQ^\tp : Q \in \SU_n\}$ \\ 
AII & $\SU_{2n}/\Sp_{2n}$ & $\{Q J_{2n} Q^\tp J_{2n}^\tp : Q \in \SU_{2n}\}$ \\
AIII & $\SU_n/\mathsf{S}\bigl(\Un_k \times \Un_{n-k} \bigr)$ & $\{Q i\diag ( I_{k},  -I_{n-k})Q^{\ast} \diag ( I_{k},  -I_{n-k}) : Q \in \SU_n\}$\\
BDI & $\SO_n(\R)/\mathsf{S}\bigl(\O_k(\R) \times \O_{n-k}(\R)\bigr)$ & $\{Q \diag (I_{k},  -I_{n-k})Q^\tp \diag (I_{k},  -I_{n-k}): Q \in \SO_n(\R)\}$\\
DIII &$\SO_{2n}(\R)/\Un_{n}$ & $\{Q J_{2n} Q^\tp J_{2n}^\tp : Q \in \SO_{2n}(\R)\}$\\
CI &$\Sp_{2n}/\Un_{n}$ & $\{Q J_{2n} Q^* J_{2n}^\tp : Q \in \Sp_{2n}\}$\\ 
CII &$\Sp_{2n}/(\Sp_{2k} \times \Sp_{2n - 2k})$ & $\{Q \diag(I_{2k}, -I_{2n - 2k}) Q^* \diag(I_{2k}, -I_{2n - 2k}) : Q \in \Sp_{2n}\}$
\end{tabu}
\caption{Cartan embeddings of symmetric spaces. Types A, BD, C are omitted since $\SU_n$, $\SO_n(\R)$, $\Sp_n$ are already Cartan embedded.} 
\label{tab:symm}
\end{table}
\end{corollary}
\begin{proof}
First note that a Cartan embedding is a manifold representation in the sense of Definition~\ref{def:rep}.
This follows either from our Theorem~\ref{thm:min} or by directly comparing the Cartan embeddings in the third column of Table~\ref{tab:symm} with their minimal faithful representations in the second column of Table~\ref{tab:mat}, and noting they are either identical or differ by a constant invertible matrix multiplied on the right. This gives us an equivalence as in Definition~\ref{def:equiv}.
\end{proof}
It is not the case that all models for symmetric spaces of compact type have faithful representations. In Table~\ref{tab:symm}, we chose our models to be the standard form $\G/\G^\sigma$ \cite{Bump13}, where $\sigma$ is an involutive automorphism of $\G$. There are alternative models used in the literature imposing other conditions like simply-connectedness. Take the simply-connected model for type BDI as example: It is the oriented Grassmannian $\SO_n(\R)/\bigl(\SO_k(\R) \times \SO_{n-k}(\R)\bigr)$, and this does not admit a faithful representation by Theorem~\ref{thm:min}.

\section{Conclusion}

The four contributions in this work are, in order of increasing specificity, (i) proposing a general theory of representations of $\G$-manifolds (Section~\ref{sec:man}); (ii) specializing to the representations of homogeneous spaces $\G/\H$ with emphasis on faithful ones (Section~\ref{sec:homo}); (iii) giving a near-complete classification of  faithfully representable $\G/\H$ when $\G$ is a simple classical group (Sections~\ref{sec:red}--\ref{sec:C}); and (iv) describing these faithful representations explicitly and proving their minimality (Section~\ref{sec:min}).

As we mentioned in Section~\ref{sec:intro}, the focus on faithful manifold representations here is natural, because faithful group and algebra representations  also  took precedence in the history of these subjects: Group representation theory grew out of Dedekind and Frobenius's exchanges about the group determinant \cite[p.~366]{history1}, and underlying it is the regular representation, a faithful representation. Likewise, Lie's original study of Lie algebras as infinitesimal transformations is, in modern lingo, a faithful representation on the space of smooth functions \cite[p.~252]{history2}.

Just as a faithful representation of a group encodes group operation as matrix multiplication and a faithful representation of an algebra  encodes Lie bracket as matrix commutator, a faithful representation of a $\G$-manifold encodes the group action as a matrix product. In applications and computations, having a matrix representation of an abstract object can be enormously beneficial, allowing for abstract operations to be concretely performed via standard numerical linear algebra. Take manifold optimization \cite{Edelman_Arias_Smith} for example, the list of candidate manifolds in this vastly popular subject is unfortunately quite limited, essentially one has just four choices for $\M$ \cite[p.~622]{Ye_Wong_Lim_2022}. But with the work in this article, the list has now been expanded to not just the $25$ manifolds in Tables~\ref{tab:mat} but infinitely many possibilities with Theorem~\ref{thm:min}.

From a purely mathematical angle, the results in Section~\ref{sec:min} show that faithful representations can shed new light on old geometry problems: We obtained the best possible effective bounds for the dimensions of Mostow--Palais embeddings in Table~\ref{tab:mp}. Prior to these results, it was only known that these dimensions are finite; there were no effective bounds. 

\subsection*{Future work} The are several immediate avenues to extend our study of faithfulness: (a) when $\G$ is one of the omitted simple classical groups $\SO^*_{2n}$, $\SU_{p, q}$, $\Sp_{p, q}$, $\SU^*_{2n}$, or an exceptional group; (b) when $\G$ is semisimple or reductive; (c) when the $\G$-action on $\M$ is nontransitive. Nevertheless, we view the study of faithful representations as a beginning. Such representations are $\G$-equivariant \emph{embeddings} into a space of matrices. So we might for instance study $\G$-equivariant \emph{immersions} or \emph{submersions} in the context of manifold representation.


\bibliographystyle{abbrv}

\end{document}